\documentclass[12pt,leqno,fleqn]{amsart}
\usepackage{amsmath,amstext,amsthm,amssymb,amsxtra}
\usepackage{txfonts,pxfonts} 
\usepackage[colorlinks,citecolor=red,hypertexnames=true]{hyperref}
\usepackage{pgf,tikz}
\usepackage{graphicx}
\usepackage{transparent}
\usepackage{mathtools}

\allowdisplaybreaks

\swapnumbers 

\numberwithin{equation}{section}

\theoremstyle{plain}
\newtheorem{theorem}[equation]{Theorem}
\newtheorem{proposition}[equation]{Proposition}
\newtheorem{corollary}[equation]{Corollary}
\newtheorem{lemma}[equation]{Lemma}

\theoremstyle{definition}

\theoremstyle{remark}
\newtheorem{remark}[equation]{Remark}
\newtheorem{example}[equation]{Example}

\makeatletter
\@namedef{subjclassname@2010}{%
  \textup{2010} Mathematics Subject Classification}
\makeatother

\def\norm#1.#2.{\lVert#1\rVert_{#2}}
\def\Norm#1.#2.{\bigl\lVert#1\bigr\rVert_{#2}}
\def\NOrm#1.#2.{\Bigl\lVert#1\Bigr\rVert_{#2}}
\def\NORm#1.#2.{\biggl\lVert#1\biggr\rVert_{#2}}
\def\NORM#1.#2.{\Biggl\lVert#1\Biggr\rVert_{#2}}

\def\ip#1,#2,{\langle #1,#2\rangle}
\def\Ip#1,#2,{\bigl\langle#1,#2\bigr\rangle}
\def\IP#1,#2,{\Bigl\langle#1,#2\Bigr\rangle}

\def\abs#1{\lvert#1\rvert}
\def\Abs#1{\bigl\lvert#1\bigr\rvert}
\def\ABs#1{\biggl\lvert#1\biggr\rvert}

\def\N{\mathbb N}

\def\R{\mathbb R}
\def\RR{\mathbb R _+}
\def\C{\mathbb C}

\def\P{\mathbf P}
\def\T{\mathbf T}
\def\TT{\mathcal T}
\def\size{\operatorname{size}}

\def\eqdef{\stackrel{\mathrm{def}}{{}={}}}
\def\iffdef {\stackrel{\mathrm{def}}{{}\Leftrightarrow{}}}

\def\ind{\textnormal{\textbf 1}}
\newcommand{\Exp}[0]{\mathbb{E}}

\begin{document}
\title[The vector valued quartile operator]{The vector valued quartile operator}

\author[T.~Hyt\"onen]{Tuomas P. Hyt\"onen}
\address{Department of Mathematics and Statistics, P.O.B.~68 (Gustaf H\"all\-str\"omin katu~2b), FI-00014 University of Helsinki, Finland}
\email{tuomas.hytonen@helsinki.fi}
\thanks{T.H. and I.P. are supported by the European Union through the ERC Starting Grant ``Analytic-probabilistic methods for borderline singular integrals''.
T.H. is also supported by the Academy of Finland, grants 130166 and 133264.}

\author[M.T. Lacey]{Michael T. Lacey}
\address{School of Mathematics \\
Georgia Institute of Technology \\
Atlanta GA 30332 }
\email{lacey@math.gatech.edu}
\thanks{M.L. supported in part by the NSF grant 0968499, and a grant from the Simons Foundation (\#229596 to Michael Lacey). }

\author[I. Parissis]{Ioannis Parissis}
\address{Department of Mathematics and Statistics, P.O.B.~68 (Gustaf H\"all\-str\"omin katu~2b), FI-00014 University of Helsinki, Finland}
\email{ioannis.parissis@gmail.com}
\begin{abstract}
	Certain vector-valued inequalities are shown to hold for a  Walsh analog of the bilinear Hilbert transform. 
	These extensions are phrased in terms of a recent notion of quartile type of a UMD (Unconditional Martingale Differences)  Banach space $ X$.  
	Every known UMD Banach space has finite quartile type, and it was recently shown that the Walsh analog of 
	Carleson's Theorem holds under a closely related assumption of finite tile type.  
	For a Walsh model of the bilinear Hilbert transform however, the  quartile type should be sufficiently 
	close to that of a Hilbert space for our results to hold.  A full set of inequalities is quantified in terms of quartile type. 
\end{abstract}

\maketitle

\section{introduction}

We are interested in topics related to the pointwise convergence of Fourier series and allied issues 
\cite{Carleson:66,MR1491450}. 
We work in the vector-valued setting, so that our Banach spaces are of UMD  type: Martingale differences 
are unconditional in these spaces 
\cite{Burkholder:83}. 
In addition, following a well-known theme in this subject, this paper concentrates on Walsh 
models  of objects with a more natural formulation in the Fourier setting.  The Walsh setting allows one to avoid 
certain technicalities in the arguments.   The martingale structure is also easier to identify, clarifying the role of UMD 
in these questions. 

Two of us \cite{12020209} have established an 
extension of the Walsh model of the Carleson Theorem on Fourier series to all known examples 
of UMD spaces.   This extension is phrased in terms of a notion of \emph{tile type}, a closely related 
notion of \emph{quartile type} being defined  below. 
Very crudely, tile type is between $ 2$ and infinity, and measures how close the UMD space is 
to a Hilbert space.  It was shown that Hilbert space has tile type $ 2$ and, that 
every UMD space which is the complex interpolation space between a Hilbert space and some other worse UMD space, 
has tile type $q\in[ 2, \infty )$.  All known examples of UMD spaces are complex interpolation spaces of this form and, thus, have finite tile type. In particular, the space $ \ell ^{p}$ has tile type  $q\in( \max \{p, p/(p-1)\}, \infty )$.
 
Mere finiteness of tile type is sufficient for the Carleson Theorem.   
Parcet-Soria-Xu have established a weaker result \cite{PSX} valid for all UMD spaces. The validity of Carleson's 
Theorem on an arbitrary UMD space is  open.  

Following  the results of \cite{MR1491450}, there has been a variety of results established that depend 
upon the ideas and techniques of Carleson's Theorem.  The bilinear Hilbert transform is perhaps the first among 
equals among these results, and it, again in the Walsh setting as identified by \cite{Thiele:quartile}, is the main object 
of present concern.  In the Fourier setting, the bilinear Hilbert transform is the operator given by 
\begin{equation*}
	(f,g) \mapsto	\textup{p.v.}\int \! f (x+y) g (x-y) \; \frac {dy} y 
\end{equation*}
We show that there are meaningful extensions of the bilinear Hilbert transform to a UMD setting. 
Moreover, using the notion of quartile type, we find that there is an intricate 
relationship between quartile type and positive results for the vector valued quartile operator. We require that the 
UMD spaces are `sufficiently close' to a Hilbert space, as measured by the quartile type.
These 
results are made precise below. 
	In addition, the Banach spaces we consider need \emph{not} be lattices, a frequent assumption 
that simplifies certain proofs.

To take advantage of the symmetries of the operator, the quartile operator is frequently studied as a trilinear form $\Lambda$,  
a sum of a product of three rank-one wave-packet projections. It is extended below in these terms. 
Take three UMD spaces, $ X_1, X_2, X_3$, and a fixed trilinear form $ \Pi \;:\; X_1\times X_2\times X_3\to \C$.  
The trilinear form  $\Lambda$  is then a sum of $ \Pi $ evaluated on certain wave-packet projections of $ f_j \in L ^{p_j} (\mathbb R _+ , X_j)$. 
A natural case of the trilinear form is to take $ X_2= X_1 ^{\ast} $ and $ X_3= \mathbb C $, so that 
$ \Pi (x,y,z)= \langle x,y \rangle \cdot z$, with the inner product denoting the duality pairing between $ X_1$ and $ X_1 ^{\ast} $.  
This is the \emph{duality trilinear form}.  

In the scalar case,  an important role is played by the set of inequalities of  the locally square integrable case.  In the UMD setting, this is played by $ q_j \le p_j < \infty $, where $ q_j$ is the quartile type of the space $ X_j$. 
We find it necessary to assume that 
$ \sum_{j} 1/q_j > 1$, which is the quantification of the triple of UMD spaces being sufficiently close to Hilbert spaces.  With the duality trilinear form,  this condition reduces to the quartile type of $X$ being $ q\in[2,4)$, that is  we can only prove results for an $ \ell ^{p}$ valued 
quartile operator in the case
$ \lvert  \frac 12 - \frac 1p\rvert < \frac 14 $. 
Again in the scalar case, the bilinear Hilbert transform satisfies a range of inequalities that permit one to break the duality in the pair 
of indices in question. There are similar phenomena in the vector valued case but the inequalities are more intricate and become restrictive 
as $  \sum_{j} 1/q_j $ approaches one.  With the duality trilinear form, there is a dichotomy between the cases of $q\in[2,3]$ and $q\in(3,4)$.

With precise definitions to follow, one of the main results of this paper is below. 

\begin{theorem}\label{t.main-local}
For $v=1,2,3$,  let $X_v$ be a UMD Banach space of quartile type $q_v$. Suppose that the quartile types satisfy
\begin{align*}
	\frac{1}{q_1}+\frac{1}{q_2}+\frac{1}{q_3}>1,
\end{align*}
and that the exponents $p_1,p_2$ and $r$ satisfy
\begin{align*}
	\frac{1}{r }=\frac{1}{p_1}+\frac{1}{p_2},\quad q_1<p_1<\infty,\quad q_2<p_2<\infty, \quad 1<r <q_3 '.
\end{align*}
Then the bilinear quartile operator extends to a bounded operator from $L^{p_1}(\RR;X_1)\times L^{p_2}(\RR;X_2)$ into $L^{r}(\RR;X_3)$, that is
\begin{align*}
	\norm B(f_1,f_2). L^{r}(\RR;X_3). \lesssim \norm f_1.L^{p_1}(\RR;X_1). \norm f_2.L^{p_2}(\RR;X_2)..
\end{align*}
\end{theorem}
Theorem \ref{t.main-local} contains the local $L^{q_1}\times L^{q_2}$ estimates for the quartile operator and can be thought of as the vector-valued analogue of the local $L^2$-theorem for the scalar bilinear Hilbert transform \cite{MR1491450}. A full set of inequalities is discussed in the concluding section of the paper. In particular, Theorem \ref{t.main-local} is a special case of our main result, Theorem \ref{thm:main}, which also contains estimates for the quartile operator `below' the quartile types.

Prabath Silva \cite[Theorem 1.9]{Silva:BP} has proved vector-valued inequalities for the bilinear Hilbert transform, concentrating 
on sequence spaces.    A special case of \cite[Theorem 1.9]{Silva:BP} concerns the  natural bilinear product $\Pi_3:\ell^R\times \ell^\infty \to \ell^R$ which is induced by the trilinear form $\Pi:\ell^R\times\ell^\infty\times \ell^{R'}\to \C$.
The trilinear operator is the pointwise product and sum.  And, it is required that $ 4/3 < R < 4$, which is similar to 
our restriction, in the duality case. \color{black} But, of course $ \ell ^{\infty } $ is not a UMD space. Thus,   
one has non-trivial vector-valued inequalities for non-UMD spaces for the bilinear Hilbert transform.  
Appropriate weighted inequalities for the bilinear Hilbert transform are not known. 
The relatively intricate nature of our and Silva's results might shed some light on this issue.
Interestingly, Silva applies his vector-valued inequalities to a question concerning the bilinear Hilbert transform, tensored with a paraproduct operator.

In \S\ref{s.quartiles}, definitions of quartiles and wave packets are recalled; in \S\ref{s.haarShifts} we discuss the $L^p$-boundedness of UMD-valued Haar-shifts, in the specific context that these operators appear within the paper; quartile type is defined and discussed in 
\S\ref{s.quartileType}; the quartile operator is discuss in \S\ref{s.quartileOperator}, while we refer the reader to 
\cite{Thiele:quartile} for an explanation of how the quartile operator is related to the bilinear Hilbert transform. In sections \S\ref{s.tree} and \S\ref{s.size} we recall some standard notions from time-frequency analysis and prove the tree lemma and the size lemma in the vector-valued setup. Section \S\ref{s.local} contains the local $L^{q_1}\times L^{q_2}$-estimates of Theorem \ref{t.main-local} while the `below quartile type' estimates are given in \S\ref{s.below} together with our main result, Theorem \ref{thm:main}. 

A few words about our notation. We write $A \lesssim B$ whenever $A\leq c B$ for some numerical constant $c>0$. Our constants typically depend on the quartile types of the Banach spaces as well as on the exponents of the $L^p$-spaces under consideration, a fact which we suppress in most places below. We write $\langle f \rangle_I\eqdef \fint_I f\eqdef \frac{1}{|I|}\int_I f$ for any locally integrable function $f$. Finally, for any dyadic interval $J$ we write $J^{(1)}$ for the dyadic parent of $J$ and $J_l,J_r$ for the left and right dyadic child of $J$, respectively. 

\section{Quartiles and wave packets}\label{s.quartiles}

We define several geometric objects, which all live in the phase space $\R_+\times\R_+$.
A quartile $P\subset \RR\times \RR$ is a dyadic rectangle of area $4$, namely 
\begin{align*}
P&=I_P\times \omega_P=I_P\times \frac{1}{|I|}[4n,4(n+1))\\
&= \bigcup_{v=0} ^3 I_P\times \frac{1}{|I_P|} [4n+v,4n+v+1) 
\\
& \eqdef \bigcup_{v=0} ^3 I_P\times \omega_{P_v}\eqdef\bigcup_{v=0} ^3 P_v.
\end{align*}
Note that $I_P$ can be written in the form $I_P=I_{P_l}\cup I_{P_r}$ where $I_{P_l}$, $I_{P_r}$ is the left and right dyadic child of $I_P$, respectively.

If $P,P'$ are quartiles, or more generally dyadic rectangles of equal area, we write $P\leq P'$ if $I_P\subset I_{P'}$ and $\omega_{P'} \subset \omega_{P}$. For any $v\in\{0,1,2,3\}$ we also define 
\begin{align*}
P\leq_v P' \iffdef P_v\leq P_v  '.
\end{align*}
A tree $\T$ is a collection of quartiles $P$ such that $P\leq T$ for some quartile $T$ which is not necessarily part of the collection. The quartile $T$ is a called \emph{a top} of the tree $\T$. Observe that the top of a tree is not uniquely defined. We will also say that a collection of quartiles $\T$ is a $u$-tree, where $u\in\{0,1,2,3\}$, if $P\leq_u T$ for all $P\in\T$ and some top $T$.

We mainly consider collections of quartiles with a certain convexity property. We will say that a collection of quartiles $\P$ is \emph{convex} if $P',P''\in\P$ and $P'\leq P\leq P''$ imply that $P\in \P$.

Here are some basic facts about trees:

\begin{list}{}{}
	\item[(i)] Suppose that $\T$ is a $v$-tree for some $v\in\{0,1,2,3\}$. Then for every $u\in\{0,1,2,3\}$ with $u\neq v$ we have that $P_{u}\cap P' _{u}=\emptyset$ for all $P,P'\in\T$ with $P\neq P'$.
	\item[(ii)] If $\T$ is a tree then by setting $\T_v\eqdef \{P\in \T: P\leq_v T\}$, $v\in\{0,1,2,3\}$, we have that $\T=\cup_{v=0} ^3 \T_v$ and each $\T_v$ is a $v$-tree.
	\item[(iii)] All tiles $P_0, P_1, P_2, P_3$ have the same time interval $I_P$ and corresponding frequency intervals $\omega_{P_0},  \omega_{P_1}, \omega_{P_2}, \omega_{P_3}$; thus $P_v=I_P\times \omega_{P_v}$ for $v\in\{0,1,2,3\}$.
\end{list}

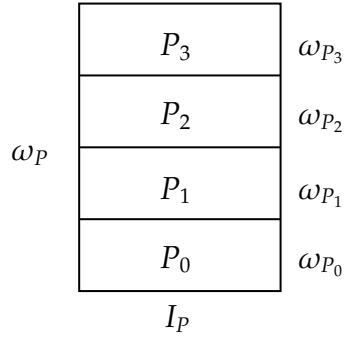
\begin{figure}[htb]
\centering
 \def\svgwidth{200pt}
\begingroup%
  \makeatletter%
  \providecommand\color[2][]{%
    \errmessage{(Inkscape) Color is used for the text in Inkscape, but the package 'color.sty' is not loaded}%
    \renewcommand\color[2][]{}%
  }%
  \providecommand\transparent[1]{%
    \errmessage{(Inkscape) Transparency is used (non-zero) for the text in Inkscape, but the package 'transparent.sty' is not loaded}%
    \renewcommand\transparent[1]{}%
  }%
  \providecommand\rotatebox[2]{#2}%
  \ifx\svgwidth\undefined%
    \setlength{\unitlength}{147.01300049bp}%
    \ifx\svgscale\undefined%
      \relax%
    \else%
      \setlength{\unitlength}{\unitlength * \real{\svgscale}}%
    \fi%
  \else%
    \setlength{\unitlength}{\svgwidth}%
  \fi%
  \global\let\svgwidth\undefined%
  \global\let\svgscale\undefined%
  \makeatother%
  \begin{picture}(1,0.62754701)%
    \put(0,0){\includegraphics[width=\unitlength]{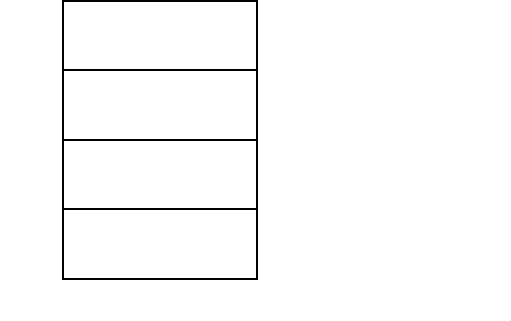}}%
    \put(0.27675933,0.12563129){\color[rgb]{0,0,0}\makebox(0,0)[lb]{\smash{$P_0$}}}%
    \put(0.27675933,0.25623199){\color[rgb]{0,0,0}\makebox(0,0)[lb]{\smash{$P_1$}}}%
    \put(0.27675933,0.39771608){\color[rgb]{0,0,0}\makebox(0,0)[lb]{\smash{$P_2$}}}%
    \put(0.27675933,0.52831677){\color[rgb]{0,0,0}\makebox(0,0)[lb]{\smash{$P_3$}}}%
    \put(0.28648946,0.01283369){\color[rgb]{0,0,0}\makebox(0,0)[lb]{\smash{$I_P$}}}%
    \put(0.53796072,0.12563129){\color[rgb]{0,0,0}\makebox(0,0)[lb]{\smash{$\omega_{P_0}$}}}%
    \put(0.53796072,0.25623199){\color[rgb]{0,0,0}\makebox(0,0)[lb]{\smash{$\omega_{P_1}$}}}%
    \put(0.53796072,0.39771608){\color[rgb]{0,0,0}\makebox(0,0)[lb]{\smash{$\omega_{P_2}$}}}%
    \put(0.53796072,0.52831677){\color[rgb]{0,0,0}\makebox(0,0)[lb]{\smash{$\omega_{P_3}$}}}%
    \put(-0.00451703,0.32890772){\color[rgb]{0,0,0}\makebox(0,0)[lb]{\smash{$\omega_P$}}}%
  \end{picture}%
\endgroup%
\caption[A quartile $P$]{A quartile $P=I_P\times \omega_P$.}
\end{figure}

We will now introduce wave packets adopted to each quartile $P$ of the time frequency plane. We first consider the Rademacher functions
\begin{align*}
	r_i(x)\eqdef \operatorname{sgn} \sin(2\pi\cdot 2^i x),\quad i\in\N,
\end{align*}
where $\N$ denotes the set of non-negative integers. If $n\in\N$ has the binary expansion $n=\sum_{i=0} ^\infty 2^i n_i$ we define the $n$-th Walsh function as
\begin{align*}
	w_n(x)\eqdef \prod_{i=0} ^\infty r_i(x)^{n_i},\quad x\in\RR.
\end{align*}
The set $\{\ind _{[0,1)}w_n\}_{n\in\N}$ is an orthonormal basis of $L^2(0,1)$.

A \emph{tile} is a dyadic rectangle $P=I\times \omega\subset \RR\times \RR$ of area $1$. Given a tile $P$ we define wave packets localized in time and frequency on $P$ as follows. If $P=I_P\times|I_P|^{-1}[n,n+1)$ we set
\begin{align*}
	w_{P}(x)\eqdef { w_{I_P,n}(x)\eqdef } \frac{1}{|I_P|^\frac{1}{2}} \ind_{I_P}(x)w_{n}(\frac{x}{|I_P|})=\frac{1}{|I_P|^\frac{1}{2}}w_{P} ^\infty(x).
\end{align*}
Note that the wave packets are $L^2$-normalized. To stress the difference we denote the corresponding $L^\infty$-normalized wave packets by
\begin{align*}
	w_{P} ^\infty(x)=\ind_{I_P}(x)w_{n}(\frac{x}{|I_P|}).
\end{align*}
For a quartile  $P=I_P\times \omega_P=\bigcup_{v=0} ^3 I_P\times \omega_{P_v}$ each $P_1,P_2,P_3$ is a tile and the previous definition applies. Thus the corresponding wave packets are 
\begin{align*}
	w_{P_v}(x)= \frac{1}{|I_P|^\frac{1}{2}} \ind_{I_P}(x)w_{n_v}(\frac{x}{|I_P|})=\frac{1}{|I_P|^\frac{1}{2}}w_{P_v} ^\infty(x), \quad v\in\{1,2,3\},
\end{align*}
where $P_v=I_P\times |I_P|^{-1}[n_v,n_v+1)$. We will not need the $0$-th wave packet $w_{P_0}$ in what follows so we deliberately omit the choice $v=0$.

The following lemma is simple but fundamental in all our considerations. It gives a description of the wave packets for quartiles of a tree in terms of the simpler functions considered in Remark \ref{rem:WalshVsHaar}. Observe that the Haar functions are given as
\begin{align*}
	h_I(x)\eqdef \frac{1}{|I|^\frac{1}{2}}\ind_I(x)r_0(\frac{x}{|I|})=w_{I,1}(x).
\end{align*}

\begin{lemma}\label{l.haarshift}
Let $u,v\in\{0,1,2,3\}$ and suppose that $\T$ is a $u$-tree with top $T$. 
Then for all $P\in\T$, we have
\begin{align*}
w_{P_v}(x)=\epsilon_{PT}\cdot w_{T_u } ^\infty(x) \cdot  w_{I_P,m} (x),
\end{align*}
for some constant factor $\epsilon_{PT}\in \{-1,+1\}$, and a fixed integer $m=m_0+2m_1\in\{0,1,2,3\}$, where $m_i\in\{0,1\}$ is defined by $m_i=u_i+v_i\mod 2$ and $u_i,v_i$ come from the dyadic expansions of $u$, $v$, respectively.
Thus $m=0$ if and only if $u=v$, in which case
\begin{align*}
w_{P_u}(x)=\epsilon_{PT}\cdot w_{T_u } ^\infty(x) \cdot  \frac{\ind_{I_P}(x)}{|I_P|^{1/2}},
\end{align*}

In particular,
\begin{align*}
\ip f,w_{P_v}, w_{P_v} = \ip f \cdot w_{T_u } ^\infty, w_{I_P,m}, \cdot w_{I_P,m}\cdot w_{T_u } ^\infty,
\end{align*}
and for $u=v$,
\begin{equation*}
  \ip f, w_{P_u},w_{P_u}=\langle f \cdot w_{T_u} ^\infty \rangle_{I_P}\cdot w_{T_u} ^\infty \cdot \ind _{I_P}.
\end{equation*}
Note that $m>0$ implies that $ w_{I_P,m}$ has mean-zero.
\end{lemma}


\begin{proof}
	We have $T_u=I_T\times\abs{I_T}^{-1}[n_T,n_T+1)$, with   $n_T \equiv u \mod 4$. Consider an element $P\in\mathbf{T}$ with $P_u=I_P\times\abs{I_P}^{-1}[n_P,n_P+1)$,   with  $n_P \equiv u\mod 4$, and let $2^{-k}\eqdef \abs{I_P}/\abs{I_T}$. Then $P_u\leq T_u$ says that
\begin{equation*}
  I_P\subseteq I_T\qquad\text{and}\qquad
  2^{-k}(n_T+1)-1\leq n_P\leq 2^{-k}n_T.
\end{equation*}
If $n_T=\sum_{i=0}^{\infty}2^i n_i$, then the unique integer value of $n_P$ in the given range is
\begin{equation*}
  n_P=\sum_{i=k}^{\infty} 2^{i-k}n_i,
\end{equation*}
which is equivalent to $ u \mod 4$ if and only if $n_k+ 2 n _{k+1}=u=u_0+2u_1$, if and only if $n_{k+i}=u_i$ for $i=0,1$.
For those values of $k$, we have $P_v=I_P\times\abs{I_P}^{-1}[n_P'+v,n_P'+v+1)$, where 
$ n'_P = \sum_{i=k+2}^{\infty} 2^{i-k}n_i$. 
Hence  
\begin{equation*}
\begin{split}
w_{P_v}
&=\frac{\textbf 1_{I_P}}{\abs{I_P}^{1/2}}w_{n_P'+v}\Big(\frac{\cdot}{\abs{I_P}}\Big)
=\frac{\textbf  1_{I_P}}{\abs{I_P}^{1/2}}w_{2^k (n_P'+v)}\Big(\frac{\cdot}{\abs{I_T}}\Big) 
  \\
  &=\frac{\textbf  1_{I_P}}{\abs{I_P}^{1/2}} w _{2^kv}\Big(\frac{\cdot}{\abs{I_T}}\Big) \prod_{i=k+2}^{\infty}r_i\Big(\frac{\cdot}{\abs{I_T}}\Big)^{n_i} 
  \\
  &=\frac{\textbf 1_{I_P}}{\abs{I_P}^{1/2}}\prod_{i=0}^{\infty}r_i\Big(\frac{\cdot}{\abs{I_T}}\Big)^{n_i}\times w _{v}\Big(\frac{\cdot}{\abs{I_P}}\Big) 
  \times
        \prod_{i=0}^{k+1}r_i\Big(\frac{\cdot}{\abs{I_T}}\Big)^{n_i} \\  
  &\overset{(*)}{=}\frac{\textbf 1_{I_P}}{\abs{I_P}^{1/2}}\prod_{i=0}^{\infty}r_i\Big(\frac{\cdot}{\abs{I_T}}\Big)^{n_i}\times w _{m}\Big(\frac{\cdot}{\abs{I_P}}\Big) 
  \times
        \prod_{i=0}^{k-1}r_i\Big(\frac{\cdot}{\abs{I_T}}\Big)^{n_i} \\  
        &=1_{I_T}w_{n_T}\Big(\frac{\cdot}{\abs{I_T}}\Big)\times\frac{\textbf 1_{I_P}}{\abs{I_P}^{1/2}} w_ {m}\Big(\frac{\cdot}{\abs{I_P}}\Big)\times
        \prod_{i=0}^{k-1}r_i\Big(\frac{\cdot}{2^k\abs{I_P}}\Big)^{n_i} 
        \\
        &=w_{T_u}^\infty\times w_{I_P,m }\times
        \prod_{i=0}^{k-1}r_i\Big(\frac{\cdot}{2^k\abs{I_P}}\Big)^{n_i} .    
\end{split}
\end{equation*}
Note that $n_k+2n_{k+1}=u$ was used in $(*)$, together with $r_i^2\equiv 1$.  In this line, we write
\begin{equation*}
  \prod_{i=0}^1 r_{k+i}\Big(\frac{\cdot}{\abs{I_T}}\Big)^{n_{k+i}}
  =\prod_{i=0}^1 r_{i}\Big(\frac{\cdot}{\abs{I_P}}\Big)^{u_i}
  =w_u\Big(\frac{\cdot}{\abs{I_P}}\Big),
\end{equation*}
and then  $w _{m}=w _{u} \cdot w _{v}$, where $m_i=u_i+v_i\mod 2$.
Note also that the last product takes a constant value on $I_P$, as $r_i$ is constant over dyadic intervals of length $2^{-i-1}$; this is our $\epsilon_{PT}$. The second claim follows from $\epsilon_{PT}^2=1$.
\end{proof}

\section{Haar shifts on a UMD Banach space} \label{s.haarShifts}
At several points below, we will need to control the $L^p(\R;X)$ norm of simple ``shifts'' of the Haar expansion of $f\in L^p(\R;X)$,  where $X$ is some UMD Banach space. Such Haar shifts have played an important role in some recent developments in Harmonic Analysis, and quite an extensive theory is available in the scalar-valued $L^p(\R)$ spaces. In the present vector-valued context, we only need relatively restricted aspects of this theory, so we choose not to develop it in a comprehensive way. We simply record for later reference the following useful estimate, where the required control is easily provided by the UMD property of $X$:
\begin{lemma}\label{lem:HaarShift}
For an interval $I$, let $I_\ell$ and $I_r$ be its left and right halves, and $I_0:=I$. For any choice of $a,b\in\{0,\ell,r\}$ and $1<p<\infty$, we have the estimate
\begin{equation*}
  \NOrm \sum_{I\in\mathcal{J}} h_{I_a}\ip h_{I_b}, f, . L^p(\R;X).
  \lesssim \NOrm \sum_{I\in\mathcal{J}} h_{I_b}\ip h_{I_b}, f, . L^p(\R;X).
  \lesssim \Norm f . L^p(\R;X). ,
\end{equation*}
for any subset $\mathcal{J}$ of the dyadic intervals.
\end{lemma}

\begin{proof}
Observe the pointwise identities $|h_I |=1_I/|I|^{1/2}$ as well as $\Exp_I 1_{I_a}=\tfrac12 1_I$ for $a\in\{\ell,r\}$, and $\Exp_I 1_I=1_I$, where $\Exp_I f:=1_I\fint_I f$ is the conditional expectation on $I$. Let $\eta_I$ stand for independent random signs indexed by the intervals $I\in\mathcal{J}$. By the unconditionality of the Haar functions, the contraction principle, Stein's inequality for the conditional expectations, and another use of the unconditionality, we have
\begin{equation*}
\begin{split}
  \NOrm \sum_{I\in\mathcal{J}} h_{I_a}\ip h_{I_b}, f, . L^p(\R;X).
  &\lesssim \Big(\Exp\NOrm \sum_{I\in\mathcal{J}} \eta_I \frac{1_{I_a}}{\abs{I_a}^{1/2}}\ip h_{I_b}, f, . L^p(\R;X). ^p\Big)^{1/p} \\
  &\leq \Big(\Exp\NOrm \sum_{I\in\mathcal{J}} \eta_I \frac{1_{I}}{\abs{I_a}^{1/2}}\ip h_{I_b}, f, . L^p(\R;X). ^p\Big)^{1/p} \\
  &\lesssim \Big(\Exp\NOrm \sum_{I\in\mathcal{J}} \eta_I \Exp_I \frac{ 1_{I_b}}{\abs{I_b}^{1/2}}\ip h_{I_b}, f, . L^p(\R;X). ^p\Big)^{1/p} \\
  &\lesssim \Big(\Exp\NOrm \sum_{I\in\mathcal{J}} \eta_I \frac{ 1_{I_b}}{\abs{I_b}^{1/2}}\ip h_{I_b}, f, . L^p(\R;X). ^p\Big)^{1/p} \\
  &\lesssim \NOrm \sum_{I\in\mathcal{J}}  h_{I_b}\ip h_{I_b}, f, . L^p(\R;X).  .
\end{split}
\end{equation*}
Note that both the unconditionality estimates and Stein's inequality relied on the UMD property of $X$.
The second estimate in the statement of the Lemma is a direct application of the unconditionality of the Haar functions.
\end{proof}

\begin{lemma}\label{lem:shiftBMO}
For an interval $I$, let $I_\ell$ and $I_r$ be its left and right halves, and $I_0:=I$. For any choice of $a,b\in\{0,\ell,r\}$, we have the estimate
\begin{equation*}
  \NOrm \sum_{I\in\mathcal{J}} h_{I_a}\ip h_{I_b}, f, . \operatorname{BMO}(\R;X).
  \lesssim \Norm f . L^\infty(\R;X). ,
\end{equation*}
for any subset $\mathcal{J}$ of the dyadic intervals.
\end{lemma}

\begin{proof}
Let $J$ be any dyadic interval. Then
\begin{equation*}
  (\ind_J-\Exp_J)\sum_{I\in\mathcal{J}} h_{I_a}\ip h_{I_b}, f,
  =\sum_{\substack{I\in\mathcal{J}\\ I_a\subseteq J}} h_{I_a}\ip h_{I_b}, f, 
  =\sum_{\substack{I\in\mathcal{J}\\ I_a\subseteq J}} h_{I_a}\ip h_{I_b}, 1_{J^{(1)}} f, .
\end{equation*}
By the previous lemma, applied to $\{I\in\mathcal{J}:I_a\subseteq J\}$ in place of $\mathcal{J}$, we get
\begin{equation*}
  \NOrm (\ind_J-\Exp_J)\sum_{I\in\mathcal{J}} h_{I_a}\ip h_{I_b}, f, . L^p(\R;X).
  \lesssim\Norm 1_{J^{(1)}} f . L^p(\R;X).
  \lesssim \abs{J}^{1/p}\Norm f . L^\infty(\R;X). .
\end{equation*}
This proves the lemma.
\end{proof}

\begin{remark}\label{rem:WalshVsHaar}
Observe that
\begin{equation*}
  w_{I,1}=h_I,\qquad w_{I,2}=\frac{h_{I_\ell}+h_{I_r}}{\sqrt{2}},\qquad w_{I,3}=\frac{h_{I_\ell}-h_{I_r}}{\sqrt{2}}.
\end{equation*}
Hence expressions of the form
\begin{equation*}
  \sum_{I\in\mathcal{J}}w_{I,m}\ip w_{I,m}, f, ,\qquad m\in\{1,2,3\},
\end{equation*}
are finite linear combinations of the expressions
\begin{equation*}
  \sum_{I\in\mathcal{J}}h_{I_a}\ip h_{I_b}, f,
\end{equation*}
considered in the above two Lemmas.
\end{remark}

\section{The quartile type of a UMD Banach space} \label{s.quartileType}
Let $v,u\in\{0,1,2,3\}$ with $u\neq v$. Let $\TT$ be a collection consisting of pairwise disjoint trees. We say that $\TT$ is $(v,u)$-good if every $\T\in\TT$ is a $v$-tree and  for any $\T,\T'\in\TT$ and $P\in\T$, $P'\in\T'$ we have that $P_u\cap P' _u=\emptyset$. When we say a property holds for all $(v,u)$ good collections we will always mean that $(v,u)$ takes all possible values in $\{0,1,2,3\}^2\setminus\{u=v\}$.

We say that a Banach space $X$ has \emph{quartile-type} $q$ if the following estimate holds uniformly for every $(v,u)$-good collection $\TT$ and every $f\in L^q(\RR;X)$:
\begin{align*}
\bigg(\sum_{\T\in\TT} \NOrm \sum_{P\in \T} \ip f,w_{P_u}, w_{P_u}.L^q(\RR;X). ^ q \bigg)^\frac{1}{q} \lesssim \norm f . L^q(\RR;X). .
\end{align*}

\begin{proposition}\label{p.quartiletype} A necessary condition for quartile-type $q$ is that $X$ is a UMD space and $q\geq 2$. If a UMD space has quartile-type $q$, it has quartile-type $p$ for all $p\in[q,\infty)$. Every Hilbert space has quartile-type $2$, and every complex interpolation space $[X,H]_\theta$, $\theta\in(0,1)$, between a UMD space $X$ and a Hilbert space $H$ has quartile-type $2/\theta$.
\end{proposition}

\begin{proof} The proof is very similar the proof of \cite[Proposition 3.1]{12020209}. For the necessity of UMD, it is enough to consider only families $\TT$ consisting of a single tree, and observe that one obtains the boundedness of Haar martingale transforms from the quartile type inequality. For the necessity of $q \geq 2$, it is enough to consider only families $\TT$, every tree in which consists of just a single quartile, and the impossibility of $q<2$ follows by elementary scalar-valued considerations.

For the main implication concerning the quartile type of interpolation spaces,
consider any $(v,u)$-good collection $\TT$. We define the operators:
\begin{equation*}
\begin{split}
  \mathcal W_{\TT}f &\eqdef \bigg\{\sum_{\P\in\T} \ip f , w_{P_u},w_{P_u} \bigg\}_{\T\in\TT},\\
  \tilde{\mathcal W}_{\TT}f &\eqdef \bigg\{w_{T_v}^\infty\sum_{\P\in\T} \ip f , w_{P_u},w_{P_u} \bigg\}_{\T\in\TT}
  =\bigg\{\sum_{\P\in\T} \ip f \cdot w_{T_v}^\infty, w_{I_P,m},w_{I_P,m} \bigg\}_{\T\in\TT},
\end{split}
\end{equation*}
where the last identity is due to Lemma~\ref{l.haarshift}.
Observe that since the collection $\TT$ is $(v,u)$-good, all the $\omega_{P_u}$ appearing in the definition of $\mathcal W_\TT$ are disjoint,  and hence the corresponding wave packets $w_{P_u}$ are pairwise orthogonal. If $H$ is a Hilbert space we thus get
\begin{align*}
  \norm \tilde{\mathcal W}_\TT f. \ell^2(\TT;L^2(\RR;H)). =\norm \mathcal W_\TT f. \ell^2(\TT;L^2(\RR;H)). \leq \norm f. L^2(\RR;H). .
\end{align*}
Now fix a UMD Banach space $X$ and a $\T\in\TT$.  By Lemma~\ref{lem:shiftBMO} and Remark~\ref{rem:WalshVsHaar}, we have
\begin{equation*}
  \Big\| \sum_{\P\in\T} \ip f \cdot w_{T_v}^\infty, w_{I_P,m},w_{I_P,m}\Big\|_{\operatorname{BMO}(\R_+;X)}
  \lesssim\| f \cdot w_{T_v}^\infty\|_{L^\infty(\R_+;X)}=\|f\|_{L^\infty(\R_+;X)},
\end{equation*}
and hence
\begin{equation*}
  \norm \tilde{\mathcal W}_\TT f. \ell^\infty(\TT;\operatorname{BMO}(\RR;X)).  \lesssim \norm f. L^\infty(\RR;X). .
\end{equation*}
If $Y=[X,H]_\theta$ for some $\theta\in(0,1)$, interpolation implies that
\begin{align*}
\norm \mathcal W_\TT f.\ell^p(\TT;L^p(\RR;Y)) . = \norm \tilde{\mathcal W}_\TT f.\ell^p(\TT;L^p(\RR;Y)) . \lesssim \norm f.L^p(\RR;X). .
\end{align*}
The last estimate holds for any $(v,u)$-good collection $\TT$; thus $Y$ is of quartile type $2/\theta$ as we wanted to prove.

The fact that quartile type $q$ implies quartile type $p\in(q,\infty)$ is proven by a similar interpolation argument, where the quartile type $q$ assumption replaces the Hilbert space $L^2$ estimate as one end of the interpolation.
\end{proof}

\section{The quartile operator} \label{s.quartileOperator}
We fix three Banach spaces $X_1, X_2, X_3$ and assume throughout the exposition that there exists a trilinear form
\begin{align*}
\Pi \;:\; X_1\times X_2\times X_3\to \C,
\end{align*}
which is 
bounded in the sense that
\begin{align*}
|\Pi(x_1,x_2,x_3)|\lesssim \|x\|_{X_1} \|x_2\|_{X_2} \|x_3\|_{X_3}.
\end{align*}
We will just write $x_1\cdot x_2 \cdot x_3$ in the place of $\Pi(x_1,x_2,x_3)$. The form $\Pi$ also induces the  bounded bilinear forms
\begin{align*}
&\Pi_1:X_2\times X_3 \to X_1 ^*,\quad \Pi_1(x_2,x_3)\eqdef \Pi(\cdot,x_2,x_3),\\
&\Pi_2:X_1\times X_3 \to X_2 ^*,\quad \Pi_2(x_1,x_3)\eqdef \Pi(x_1,\cdot,x_3),\\
&\Pi_3:X_1\times X_2 \to X_3 ^*,\quad \Pi_3(x_1,x_2)\eqdef \Pi(x_1,x_2,\cdot).
\end{align*}
We will also use the notation $x_1\cdot x_2$, say, with the obvious meaning. 

The quartile operator is the bilinear operator
\begin{align*}
B(f_1,f_2)(x)\eqdef \sum_{P\mbox{ \tiny quartile}}  \frac{1}{|I_P|^\frac{1}{2}}\ip f_1,w_{P_1},  \cdot \ip f_2,w_{P_2}, w_{P_3}(x),
\end{align*}
where $f_1\in L^{p_1}(\RR;X_1), f_2\in L^{p_2}(\RR;X_2)$. The associated trilinear form is 
\begin{align*}
\Lambda(f_1,f_2,f_3)\eqdef \sum_{P\mbox{ \tiny quartile}} \frac{1}{|I_P|^\frac{1}{2}}\ip f_1,w_{P_1}, \cdot  \ip f_2,w_{P_2}, \cdot \ip f_3,w_{P_3}, ,
\end{align*}
with $f_3\in L^{p_3}(\RR;X)$.

It is again associated with three bilinear forms via
\begin{equation*}
  \Lambda(f_1,f_2,f_3)
  =\langle f_1, B_1(f_2,f_3)\rangle
  =\langle f_2, B_2(f_1,f_3)\rangle
  =\langle B_3(f_1,f_2),f_3\rangle,
\end{equation*}
where $B_3=B$ is the original quartile operator.

Some remarks are in order. The sums defining the bilinear quartile operator and the associated trilinear form extend through all the quartiles of the time-frequency plane. We make sense of these operators by initially considering them acting on \emph{finite dyadic step functions}, i.e., compactly supported functions on the real line which are constant on dyadic intervals of length $2^{-N}$ for some integer $N$.

\begin{lemma}
If all $f_i$ are finite dyadic step functions, then the series defining $\Lambda(f_1,f_2,f_3)$ converges absolutely.
\end{lemma}



\begin{proof}
By linearity, we may assume that $f_i=\ind_{J_i}$ for some dyadic interval $J_i$. Let $P=I\times\abs{I}^{-1}[4n,4(n+1))$, so that $w_{P_i}=w_{I,4n+i}$, and
\begin{equation*}
  \ip f_i, w_{P_i},
  =\frac{1}{|I|^{1/2}}\int_{I\cap J_i}w_{4n+i}\Big(\frac{x}{\abs{I}}\Big)dx
\end{equation*}
is nonzero only if $I\supsetneq J_i$, since $w_{4n+i}(\cdot/\abs{I})$ has mean zero on $I$. If $I=J_i^{(k)}$ and $4n+i=2^k a+b$, $0\leq b<2^k$, then
\begin{equation*}
  \int_{I\cap J_i}w_{4n+i}\Big(\frac{x}{\abs{I}}\Big)dx
  =\int_{J_i}w_{2^k a+b}\Big(\frac{x}{2^k\abs{J_i}}\Big)dx
  =w_b\Big(\frac{c(J_i)}{2^k\abs{J_i}}\Big)\int_{J_i}w_{a}\Big(\frac{x}{\abs{J_i}}\Big)dx,
\end{equation*}
since $w_b(\cdot/2^k\abs{J_i})$ is constant on $J_i$ for $b<2^k$. The last integral is zero unless $a=0$, i.e., unless $4n+i<2^k=\abs{I}/\abs{J_i}$. In summary, we conclude that $\ip f_i, w_{P_i},$, where $P=I\times\abs{I}^{-1}[4n,4(n+1))$, can be nonzero only if $I\supsetneq J_i$ and $4n+i<\abs{I}/\abs{J_i}$, and in this case
\begin{equation*}
  |\ip f_i, w_{P_i},|\leq\|f_i\|_1\|w_{P_i}\|_\infty=\frac{|J_i|}{|I|^{1/2}}.
\end{equation*}
Thus
\begin{equation*}
\begin{split}
  \sum_{P\text{ quartile}}\frac{1}{|I_P|^{1/2}} &|\ip f_1, w_{P_1}, \ip f_2, w_{P_2}, \ip f_3, w_{P_3}, |
  \leq\sum_{\substack{I\in\mathcal{D}\\ I\supset J_1\cup J_2\cup J_3}}\sum_{0\leq n<4^{-1}\abs{I}/\max_i\abs{J_i}}
    \frac{\abs{J_1}\abs{J_2}\abs{J_3}}{\abs{I}^2} \\
   &\leq \frac{\abs{J_1}\abs{J_2}\abs{J_3}}{\max_i\abs{J_i}}\sum_{\substack{I\in\mathcal{D}\\ I\supset J_1\cup J_2\cup J_3}}\frac{1}{\abs{I}} 
      \leq \frac{\abs{J_1}\abs{J_2}\abs{J_3}}{(\max_i\abs{J_i})^2}\leq \min_i\abs{J_i}.
\end{split}
\end{equation*}
This proves the absolute convergence.
\end{proof}

Thus, for finite dyadic step functions $f_i$, $\Lambda(f_1,f_2,f_3)$ is well defined, and it is the unconditional limit of the sums over finite collections $\P$ of quartiles. Henceforth, we concentrate on estimating such an arbitrary finite sum
\begin{equation*}
  \Lambda_{\P}(f_1,f_2,f_3)=\sum_{P\in\P}\frac{1}{|I_P|^\frac{1}{2} }\ip f_1,w_{P_1}, \ip f_2,w_{P_2}, \ip f_3,w_{P_3},.
\end{equation*}
	
This is just a qualitative assumption though since our estimates do not depend on the size of the collection of quartiles we are considering. Furthermore, we will assume that the collection $\P$ is convex. This is done by noting that we actually prove all our estimates for the larger operator
\begin{align*}
|\Lambda_T|(f_1,f_2,f_3) =\sum_{P\in\P}\frac{1}{|I_P|^\frac{1}{2} }\abs{\ip f_1,w_{P_1}, \ip f_2,w_{P_2}, \ip f_3,w_{P_3},}.
\end{align*}
Thus, starting with a finite collection $\P$ we only need to add finitely many quartiles to make the collection convex and our operator only get bigger. We can and will therefore assume throughout the exposition that the quartile operator and the associated trilinear form are defined on a finite, convex, but otherwise arbitrary collection of quartiles $\P$.

If $\P$ is a collection of quartiles, $v\in\{1,2,3\}$ and $1<p<\infty$, then the $(v,p)$-size of $\P$ is defined as
\begin{align*}
\size_{v,p}(\P)\eqdef \sup_{\T\subset \P} \bigg(\frac{1}{|I_T|} \int \Abs{\sum_{P\in\T} \ip f_v,w_{P_v},w_{P_v} (x)}^p dx \bigg)^\frac{1}{p},
\end{align*}
where the supremum is taken over all $u$-trees inside $\P$ with $0\leq u\neq v\leq 3$.
The following lemma shows that the definitions of size for different $p$'s are equivalent up to numerical constants. This is essentially due to the fact that the size is a BMO-norm type of quantity so the John-Nirenberg inequality applies. For $1\leq p<\infty$ we denote by
\begin{align*}
	\norm f.\operatorname{BMO}_p(\RR;X).\eqdef \sup_{I\in\mathcal D} \bigg(\frac{1}{|I|}\int_I \Abs {f(x)- \langle f \rangle _I }^p dx \bigg)^\frac{1}{p},
\end{align*}
the dyadic BMO-norm of $f$ with respect to the exponent $p$. We have:
\begin{lemma}\label{l.bmosize}Let $\P$ be any collection of tiles. For all $v\in\{1,2,3\}$ and $p,q\in[1,\infty)$ we have
\begin{equation*}
  \size_{v,p}(\P)\eqsim_{p,q}\size_{v,q}(\P).
\end{equation*}
\end{lemma}

\begin{remark} In view of Lemma \ref{l.bmosize}, all the sizes $\size_{v,p}(\P)$ are equivalent, up to numerical constants, for all $1\leq p<\infty$. However, it will be more convenient throughout the paper to set
	\begin{align*}
		\size_v(\P)\eqdef \size_{v,q_v}(\P),
	\end{align*}
whenever the underlying space $X_v$ has quartile type $q_v$.
\end{remark}

\begin{proof} Let us fix some $v\in\{1,2,3\} $.  By Lemma \ref{l.haarshift} we have that for any $u$-tree $\T$, $u\in\{0,1,2,3\}\setminus\{v\}$, the wave-packet $w_{P_v}$, $P\in\T$, is described as
\begin{align*}
w_{P_v} =\epsilon_{PT} w_{I_P,m} \cdot w_{T_u } ^\infty,
\end{align*}
where $m\in\{1,2,3\}$. Depending on the value of $m$ we have the cases
\begin{align*}
w_{I_P,m} \in\{ h_{I_P},\frac{1}{\sqrt 2}(h_{I_{P_l}}+h_{I_{P_r}}),\frac{1}{\sqrt{2}}(h_{I_{P_l}}-h_{I_{P_r}})\},\quad m=1,2,3,\quad\mbox{correspondingly}.
\end{align*}
These functions were already considered in Remark \ref{rem:WalshVsHaar}. Since we have fixed $v$, the value of $m$ depends only on the type $u$ of the tree that we are considering. Fixing $m$ implicitly fixes the value of $u$ so we will prove the lemma by considering all the possible values of $m$. 

Fix a $u$-tree $\T\subset\P$ and write
\begin{equation*}
  \sum_{P\in\T}\ip f, w_{P_v}, w_{P_v}
  =w^\infty_{T_u}\sum_{P\in\T}\ip f w^\infty_{T_u}, w_{I_P,m}, w_{I_P,m}\eqdef w^\infty_{T_u}\tilde{f}.
\end{equation*}
We will estimate the dyadic BMO norm of $\tilde{f}$. Let $J$ be a dyadic interval. Then
\begin{equation}\label{eq:sizeVsBMO}
  (\ind_J-\Exp_J)\tilde{f}
  =\sum_{\substack{P\in\T \\ I_P\subseteq J}}\ip f w^\infty_{T_u}, w_{I_P,m}, w_{I_P,m}
    +\ind_{\{2,3\}}(m)\ind_J\sum_{\substack{P\in\T \\ I_P=J^{(1)}}}\ip f w^\infty_{T_u}, w_{I_P,m}, w_{I_P,m},
\end{equation}
where the second sum contains at most one term, and occurs only for $m\in\{2,3\}$. Here, remember that $J^{(1)}$ denotes the dyadic parent of $J$.

Let $\T^*(J)$ be the collection of maximal quartiles $P\in\T$ with $I_P\subseteq J$. Since $\omega_T\subseteq\omega_P$ for all $P\in\T$, maximality implies that the time intervals $I_P$ of $P\in\T^*(J)$ are pairwise disjoint. Thus
\begin{equation*}
   \sum_{\substack{P\in\T \\ I_P\subseteq J}}\ip f w^\infty_{T_u}, w_{I_P,m}, w_{I_P,m}
   =w^\infty_{T_u}\sum_{R\in\T^*(J)}\sum_{\substack{P\in\T \\ P\leq R}}\ip f, w_{P_v},w_{P_v},
\end{equation*}
and by disjointness
\begin{equation*}
\begin{split}
  \Big\|\sum_{\substack{P\in\T \\ I_P\subseteq J}}\ip f w^\infty_{T_u}, w_{I_P,m}, w_{I_P,m}\Big\|_{L^p(\R_+;X_v)}^p
  &=\sum_{R\in\T^*(J)}\Big\|\sum_{\substack{P\in\T \\ P\leq R}}\ip f, w_{P_v},w_{P_v}\Big\|_{L^p(\R_+;X_v)}^p \\
  &\leq\sum_{R\in\T^*(J)}\size_{p,v}(\P)^p\abs{I_R}\leq\size_{p,v}(\P)^p\abs{J},
\end{split}
\end{equation*}
where the estimate by size follows from the fact that each $\{P\in\T:P\leq R\}\subseteq\P$ is another $u$-tree with top $R$, as one easily checks.

For the second term in \eqref{eq:sizeVsBMO}, we just observe that
\begin{equation*}
  \abs{\ip fw^\infty_{T_u},w_{I_P,m},}
  =\abs{\ip f, w_{P,v},}\leq\size_{p,v}(\P)\abs{I_P}^{1/2},\qquad\abs{w_{I_P,m}}\leq\abs{I_P}^{-1/2},
\end{equation*}
which gives the same estimate for the $L^p(\R_+;X_v)$ norm for this second term. Altogether, we have
\begin{equation*}
  \|(\ind_J-\Exp_J)\tilde{f}\|_{L^p(\R_+;X_v)}\leq 2\size_{p,v}(\P)\abs{J}^{1/p},
\end{equation*}
and hence by the John--Nirenberg inequality that
\begin{equation*}
  \|(\ind_J-\Exp_J)\tilde{f}\|_{L^q(\R_+;X_v)} \lesssim_{p,q} \size_{p,v}(\P)\abs{J}^{1/q}.
\end{equation*}
Specializing this to $J=I_T\supseteq\operatorname{supp}\tilde{f}$ gives
\begin{equation*}
  \|\tilde{f}\|_{L^q(\R_+;X_v)}\lesssim_{p,q}\size_{p,v}(\P)\abs{I_T}^{1/q}+\abs{I_T}^{1/q}\fint_{I_T}\abs{\tilde{f}}dx\lesssim \size_{p,v}(\P)\abs{I_T}^{1/q}.
\end{equation*}
Observing that
\begin{equation*}
  \abs{I_T}^{-1/q}\|\tilde{f}\|_{L^q(\R_+;X_v)}=\abs{I_T}^{-1/q}\Big\|\sum_{P\in\T}\ip f, w_{P_v},w_{P_v}\Big\|_{L^q(\R_+;X_v)}
\end{equation*}
and taking the supremum over all $u$-trees $\T\subseteq\P$ for all $u\neq v$ completes the proof.
\end{proof}

\section{The tree lemma}\label{s.tree}
In this section we prove the basic estimate where the quartile operator is considered over a single tree. We will need two simple auxiliary lemmas. The first gives an estimate of the randomized projections on non-overlapping tiles by the corresponding size.
\begin{lemma}\label{l.random_size} For a positive integer $N$ and for $j=0,1,2,\ldots,N-1$ denote by $\alpha_j=e^{2\pi i \frac{j}{N}}$ the $N$-the roots of unity. Suppose that the random variables $\eta_P$ take the values $\alpha_0,\ldots,\alpha_{N-1}$ on every tile $P$ with equal probability and independently of other tiles. Let $0\leq u \leq 3$ and $v\in\{1,2,3\}\setminus\{ u\}$, and let $\T$ be a $u$-tree. For $1<p<\infty$ we have
	\begin{align*}
		\bigg(\int_{I_T} \Exp\Abs{\sum_{P\in\T}\eta_P  \ip f_v,w_{P_v},  \frac{ \ind_{I_P}(x)}{|I_P|^\frac{1}{2}} }^p\bigg)^\frac{1}{p} \lesssim_{N,p} \size_{v,p}(\T)|I_T|^\frac{1}{p}.
	\end{align*}
\end{lemma}

\begin{proof} Observing that for any quartile $P$ we have $w_{P_v} w_{P_v}=\frac{1}{|I_P|}\textbf 1_{I_P}$ we write
\begin{align*}
 \Exp\Abs{\sum_{P\in\T}\eta_P  \ip f_v,w_{P_v},  \frac{ \ind_{I_P}(x)}{|I_P|^\frac{1}{2}} }^p&= \Exp\Abs{\sum_{P\in\T}\eta_P  \ip f_v,w_{P_v},   w_{P_v} (x)w_{P_v}(x) |I_P|^\frac{1}{2} }^p
\\
&\lesssim_{N,p} \max_{P\in\T} \big(|I_P|^\frac{1}{2}|w_P(x)|\big)^p \Exp\Abs{\sum_{P\in\T}\eta_P  \ip f_v,w_{P_v},   w_{P_v} (x)}^p
\end{align*}
by the Kahane contraction principle. Noting that $\max_{P\in\T} \big(|I_P|^\frac{1}{2}|w_P(x)|\big)\lesssim 1$ we get
\begin{align*}
\int_{I_T} \Exp\Abs{\sum_{P\in\T}\eta_P  \ip f_v,w_{P_v},  \frac{ \ind_{I_P}(x)}{|I_P|^\frac{1}{2}} }^pdx &\lesssim_N \int_{I_T} \Exp\Abs{\sum_{P\in\T}\eta_P  \ip f_v,w_{P_v},   w_{P_v} (x)}^p dx
\\ 
 &\lesssim \Exp \int_{I_T}\Abs{\sum_{P\in\T}  \ip f_v,w_{P_v}, w_{P_v}(x) }^p dx \lesssim |I_T|\size_{v,p}(\T) ^p,
\end{align*}
where the second approximate inequality follows by the UMD property of $X_v$ and the fact that  $\sum_{P\in\T} \eta_P \ip f_v,w_{P_v}, w_{P_v}(x) $ is a martingale transform of the function $\sum_{P\in\T}  \ip f_v,w_{P_v}, w_{P_v}(x) $. 
\end{proof}

For the case where we project on overlapping tiles we use:

\begin{lemma}\label{l.convex} Let $\mathcal J$ be a collection of dyadic intervals on the real line that is convex in the sense that for any three dyadic intervals $I',I,I''$ with $I',I''\in\mathcal J$, $I'\subseteq I\subseteq I''$ implies that $I\in\mathcal J$. For some Banach space $X$ let $f:\bigcup_{J\in\mathcal J}J \to X$ satisfy
\begin{align*}
\abs{\langle f\rangle_{J}}\leq  \lambda\quad\mbox{for all}\quad J\in\mathcal J,
\end{align*}
for some $\lambda>0$. There exists a function $g:\bigcup_{J\in\mathcal J}J \to X$ such that $\langle g\rangle_{J}=\langle f \rangle_{J}$ for all $J\in\mathcal J$ and $\norm g.L^\infty(\bigcup_{J\in\mathcal J}J ;X). \leq 3\lambda.$
\end{lemma}

\begin{proof} Let us call a dyadic interval $E\in\mathcal J$ \emph{exceptional} if it has at least one dyadic child $B\notin \mathcal J$ and denote by $\mathcal E$ the collection of all exceptional $E\in\mathcal J$. Call a dyadic interval $B\notin \mathcal J$ \emph{bad} if it is the child of some exceptional $E=E(B)\in\mathcal J$. It is obvious that all the bad dyadic intervals are pairwise disjoint. Now we define the function $g:\bigcup_{J\in\mathcal J}J \to X$ as follows. If $B$ has a dyadic sibling $J\in\mathcal J$ then we set $g(x)\eqdef \langle f\rangle_{B}$ for $x\in B$. We have
\begin{align*}
	\langle f\rangle_{E(B)}=\frac{1}{2}\langle f\rangle_B+\frac{1}{2}\langle f\rangle_J\Rightarrow \langle f\rangle_{B}=2	\langle f\rangle_{E(B)}-\langle f\rangle_J\Rightarrow | \langle g\rangle_B|=| \langle f\rangle_B| \leq 3\lambda.
\end{align*}
If $B$ has a bad sibling $B_1$ we set $g(x)\eqdef\langle f \rangle_{E(B)}$ for $x\in E(B)=B\cup B_1$. Then
\begin{align*}
\langle g\rangle_B=\langle g\rangle_{B_1}=\langle f\rangle_{E(B)}\Rightarrow |\langle g\rangle_B|\leq \lambda.
\end{align*}
If $x\notin \bigcup_{B\mbox{ \tiny bad}}B$ then just define $g(x)\eqdef f(x)$.

Now let $J\in\mathcal J$. There are two possibilities. If $J$ does not intersect any of the bad intervals, in which case $g\equiv f$ on $J$, we have that $\langle g\rangle_J=\langle f\rangle_J$. If, on the other hand, $J\cap B\neq \emptyset $ for some bad $B$ then convexity implies that $B\subsetneq J$. Now the fact that $\langle g\rangle_B=\langle f \rangle_B$ for all bad intervals $B$ together with the non-intersecting case just considered easily implies that $\langle g\rangle_J=\langle f \rangle_J$ as we wanted.

 For the $L^\infty$-bound note that if $x\in \cup_{B\mbox{ \tiny bad}} B$ then $x\in B'$ for some bad $B'$ and $|g(x)|=|\langle g \rangle_{  B'}|\leq 3\lambda$. On the other hand, for almost every $x\notin \cup_{B\mbox{ \tiny bad}}B$ there is a decreasing sequence  $I_0\supset I_1\supset \cdots\supset I_k\supset \cdots$ of nested dyadic intervals which belong to $\mathcal J$ and  contain $x$ so that
\begin{align*}
	|g(x)|=\lim_{|I_k|\to 0}|\langle g\rangle_{I_K}| = \lim_{|I_k|\to 0}| \langle f\rangle_{I_k}|\leq \lambda.
\end{align*}
Combining the previous cases we get $\norm g.L^\infty(\RR;X).\leq 3\lambda.$
\end{proof}

\begin{remark} Suppose that $\T$ is a convex tree and consider the collection of dyadic intervals 
	\begin{align*}
		\mathcal J_{\T}\eqdef \{I_P:P=I_P\times \omega_P\in\T\}.
	\end{align*}
Now suppose that $I$ is some dyadic interval and $I',I''\in \mathcal J_T$ with $I'\subseteq I \subseteq I''$. Let $P',P''\in\T$ so that $I_{P'}=I'$ and $I_{P''}=I''$. We have $|\omega_{P''}|\leq \frac{4}{|I|}\leq |\omega_{P'}|$ and $\omega_{P''}\subseteq \omega_{P'}$. Thus there is a dyadic interval $\omega$ of length $4/|I|$ such that the quartile $P\eqdef I\times \omega$ satisfies $P'\leq P \leq P''$. We conclude that $P\in \T$ thus $I\in\mathcal J_\T$. This means that a  convex tree induces a convex collection of dyadic intervals in the sense of Lemma \ref{l.convex}.
\end{remark}

\begin{lemma}\label{l.tree}If $\T$ is a convex tree then
	\begin{align*}
		\abs{\Lambda_\T(f_1,f_2,f_3)}\lesssim |I_T| \prod_{v=1} ^3 \size_v(\T).
	\end{align*}
\end{lemma}

\begin{proof} Let $\T$ be a convex $u$-tree for some $u\in\{1,2,3\}$. Recall from Lemma \ref{l.haarshift} that we have
\begin{align*}
\frac	{\ip f_u,w_{P_u},}{|I_P|^\frac{1}{2}}=\frac{\ip f_u w_{T_u} ^\infty,\textbf 1_{I_P},}{|I_P|}=\langle  w_{T_u} ^\infty  f_u\rangle_{I_P}\eqdef \langle \tilde   f_u\rangle_{I_P},
\end{align*}	
for some unimodular function $w_{T_u} ^\infty$ which depends only on the tree $\T$. Let $\mathcal J\eqdef \{I_P:P\in\T\}$. We clearly have $|\langle \tilde f_u\rangle_J|\leq \size_u(\T)$ for all $J\in \mathcal J$. Furthermore, the fact that $\T$ is a convex tree implies that the family $\mathcal J$ is convex in the sense of Lemma \ref{l.convex}. By the same lemma we conclude that there is a function $g_u:\bigcup_{J\in\mathcal J}J \to X_u$ with $\langle g_u\rangle_{I_P}=\langle\tilde f_u \rangle_{I_P}$ for all $P\in\mathcal \T$ and $\norm g_u. {L^\infty(\cup_{P\in\mathcal T}I_P;X)}. \leq 3 \size_u(\T)$. For convenience extend $g_u$ to be identically zero on $I_T\setminus \cup_{P\in\mathcal \T}I_P$.

For $\{b,c\}=\{1,2,3\}\setminus\{u\}$ we have the following identity:
\begin{align*}
\abs{\Lambda_\T(f_1,f_2,f_3)}&=\Abs{\sum_{P\in\T}\frac{1}{|I_P|^\frac{1}{2}} \ip f_1,w_{P_1},\cdot \ip f_2,w_{P_2},\cdot \ip f_3,w_{P_3},}
\\
&=\ABs{\sum_{P\in\T} \langle  \tilde  f_u \rangle_{I_P} \cdot   \ip f_b,w_{P_b}, \cdot \ip f_c,w_{P_c},}
\\
&=\ABs{\int_{I_T} \sum_{P\in\T}g_u(x)\cdot \frac{\ip f_b,w_{P_b},}{|I_P|^\frac{1}{2}}  \cdot \frac{\ip f_c,w_{P_c}, }{|I_P|^\frac{1}{2} }\textbf 1_{I_P}(x) dx}
	\\
	&=\ABs{\int_{I_T}g_u(x)  \cdot \Exp\bigg [ \bigg(\sum_{P\in\T}\eta_P  \ip f_b,w_{P_b}, \frac{\textbf 1_{I_P}(x)}{|I_P|^\frac{1}{2}}  \bigg) \cdot \bigg(\sum_{P\in\T}\eta_P  \ip f_c,w_{P_c}, \frac{\textbf 1_{I_P}(x)}{|I_P|^\frac{1}{2}}  \bigg) \Bigg ] dx},
\end{align*}
where we choose the value of $\eta_P$ for each tile $P$ to be $\pm 1$ with equal probability and independently of other tiles. Using Cauchy-Schwarz twice we get
\begin{align*}
	\abs{\Lambda(f_1,f_2,f_3)}&\leq \norm g_u .L^\infty(\RR;X). \bigg(\int_{I_T} \Exp\Abs{\sum_{P\in\T}\eta_P  \ip f_b,w_{P_b},  \frac{\textbf 1_{I_P}(x)}{|I_P|^\frac{1}{2}} }^2\bigg)^\frac{1}{2} \bigg(\int_{I_T}\Exp\Abs{\sum_{P\in\T}\eta_P  \ip f_c,w_{P_c},  \frac{\textbf 1_{I_P}(x)}{|I_P|^\frac{1}{2}} }^2\bigg)^\frac{1}{2} 
	\\
	&\lesssim |I_T|\prod_{v=1} ^3 \size_v(\T),
\end{align*}
where in the last inequality we have used Lemmas \ref{l.random_size} and \ref{l.bmosize}, and the fact that $\norm g_u.L^\infty(\RR;X_1).\leq 3 \size_u(\T)$.

If $\T$ is a $0$-tree we argue as follows. Let $\eta_P$ take the values $\alpha_0,\alpha_1,\alpha_2$, i.e., the third roots of unity, on each tile, with equal probability and independently of other tiles. We write
\begin{align*}
	\abs{\Lambda_\T(f_1,f_2,f_3)}&=\ABs{\int_{I_T}\prod_{v=1} ^3 \Exp\bigg(\sum_{P\in\T}\eta_P  \ip f_v,w_{P_v}, \frac{\textbf 1_{I_P}(x)}{|I_P|^\frac{1}{2}}  \bigg)   dx}
	\\
	&\leq \prod_{v=1} ^3\bigg(\frac{1}{|I_T|}\int_{I_T}\Exp \ABs{\sum_{P\in\T}\eta_P  \ip f_v,w_{P_v}, \frac{\textbf 1_{I_P}(x)}{|I_P|^\frac{1}{2}}  } ^3  dx\bigg)^\frac{1}{3},
\end{align*}
by a double application of H\"older's inequality with exponents $p_1=p_2=p_3=3.$ Applying Lemmas \ref{l.random_size} and \ref{l.bmosize}, we get
\begin{align*}
		\abs{\Lambda_\T(f_1,f_2,f_3)}\leq|I_T| \prod_{v=1} ^3 \size_{v,3}(\T)\simeq |I_T|\prod_{v=1} ^3 \size_v(\T).
\end{align*}
Since any convex tree splits as $\T=\T_0\cup\T_1\cup\T_2\cup \T_3$, where $\T_u$ is a convex $u$-tree, the proof is complete.
\end{proof}

\begin{remark} The previous estimate remains valid for the operator $|\Lambda_T|$. One way to see this in all cases is to write
\begin{align*}
\abs{\Lambda_\T}(f_1,f_2,f_3)&=\sum_{P\in\T}\frac{1}{|I_P|^\frac{1}{2}} \Abs{\ip f_1,w_{P_1},\cdot \ip f_2,w_{P_2},\cdot \ip f_3,w_{P_3},}
\\
&=\Abs{\sum_{P\in\T}\frac{\epsilon_P}{|I_P|^\frac{1}{2}} \ip f_1,w_{P_1},\cdot \ip f_2,w_{P_2},\cdot \ip f_3,w_{P_3}, }
\end{align*}
for some choice of complex sign $\epsilon_P$. The reader can easily check that the presence of an arbitrary complex sign in the sum above does not affect the estimates in Lemma \ref{l.tree}.
\end{remark}

\section{The size lemma}\label{s.size}

In the previous section, we saw how to control our trilinear form $\Lambda$, when the summation is restricted to a convex tree $\T$. But eventually we need to consider $\Lambda$ over an arbitrary convex collection $\P$. The content of the next proposition is that we can always extract a sequence of trees from any such $\P$, in such a way that the size of the remaining part of $\P$ becomes strictly smaller than the original size. This will then allow an iteration, by which all of $\P$ is decomposed into trees in a controlled manner.

\begin{proposition} Fix some $v\in\{1,2,3\}$, and let $X_v$ be a UMD space with quartile-type $q_v$. Then every finite, convex set of quartiles $\P$ has a disjoint decomposition
	\begin{align*}
		\P=\P_{\operatorname{small}}\cup\bigcup_j \T_j,
	\end{align*}
where $ \P_{\operatorname{small}}$ is a convex collection of tiles,  each $\T_j$ is a convex tree, and
\begin{align*}
	\size_{v} (\P_{\operatorname{small}})\leq 2^{-\frac{1}{q}} \size_{v}(\P) ,\quad \sum_{j}|I_{T_j}|\leq \size_{v} (\P) ^{-q} \|f\|_{L^{q_v}(\RR;X)} ^{q_v},
\end{align*}
where the $v$-size is defined with respect to the function $f\eqdef f_v\in L^{q_v}(\RR;X_v)$. Here the underlying space $X_v$ has quartile type $q_v$ so we use $\size_v(\P)=\size_{v,q_v}(\P)$.
\end{proposition}

\begin{proof}For every tree $\T $ and $u \in \{ 0,1,2,3 \}\setminus\{v\}$, let
\begin{align*}
	\Delta_u (\T)^{q}\eqdef \frac{1}{|I_T|}\int\Abs{ \sum_{P\in \T_u} \ip f,w_{P_v}, w_{P_v} }^{q} dx,
\end{align*}
where $\T_u\eqdef \{P\in\T:P\leq_u T\}$ is the $u$-tree supported by the same top.

Let $\mathcal E_v \eqdef \size_{v,q	}(\P)$. For each (momentarily fixed) $u\in\{0,1,2,3\}\setminus\{v\}$, we iterate the following procedure:
Consider all maximal trees in $\P$ among the ones with $\Delta_u (\T)>\mathcal{E}_v 2^{-\frac{1}{q}}$. 
Among them let $\T^u_{1}$ be the one whose top $T^u_{1}$ has frequency interval $\omega_{T^u_{1}}$ with \emph{minimal} center $c(\omega_{T^u_{1}})$
if $u>v$, or \emph{maximal} center  $c(\omega_{T^u_{1}})$ if $u<v$.  Replace $\P$ by $\P\setminus \T_1$ and iterate. When no further trees can be found subject to the criterion that  $\Delta_u (\T)>\mathcal{E}_v 2^{-\frac{1}{q}}$ for the given $u$, we shift to the next value of $u$, and go through the same iteration. When we have completed this procedure for all values of $u\neq v$, the remaining collection $\P_{\operatorname{small}}$ satisfies
\begin{align*}
\Delta_u(\T)	\leq 2^{-\frac{1}{q}}\size_{v} (\P)\quad\mbox{for all}\quad u\in\{0,1,2,3\}\setminus\{v\},
\end{align*}
for all trees $\T\subset \P_{\operatorname{small}}$. By definition this means that we have $\size_{v}(\P)\leq2^{-\frac{1}{q}}\size_{v}(\P)$. Also, we have extracted trees $\T^u_{j}$, where $u\in\{0,1,2,3\}\setminus\{v\}$ and $j$ is a running index, so that $\Delta_{u} ( \T^u_j)>\mathcal E _v2^{-\frac{1}{q}}$.

For the collection $\{\T^u_{j}\}_j$ we estimate the sum over the top intervals as follows:
\begin{align*}
	\sum_{j}|I_{T_{u,j}}|& \leq \frac{2}{\mathcal E_v ^{q} } \sum_j \NOrm \sum_{P\in \T ^u_{j,u}}\ip f,w_{P_{v}},w_{P_{v}} .L^{q}(\RR;X). ^{q},
	\lesssim \frac{1}{\mathcal E_v ^{q} } \sum_{s=0}^1 \sum_j \NOrm \sum_{P\in \T ^u_{j,u}(s)}\ip f,w_{P_{v}},w_{P_{v}} .L^{q}(\RR;X). ^{q},
\end{align*}
where $\T^u_{j,u}$ is a $u$-tree, contained in the maximal tree $\T^u_j$, which realizes the size selection condition. We set
\begin{equation}\label{eq:splitLength}
  \T^u_{j,u}(s)\eqdef\{P\in\T^u_{j,u}:\log_2\abs{\omega_P}\equiv s\mod 2\},\qquad s=0,1.
\end{equation}
We fix $s\in\{0,1\}$ and suppress it from the notation, writing $\TT_u\eqdef \{\T^u_{j,u}(s)\}_j\eqdef \{T_{j,u}\}_j$, which is a collection of $u$-trees.

In order to apply the quartile-type hypothesis we need to check that $\TT_u$ is $(u,v)$-good. To that end, we split the treatment into two cases.
\subsubsection*{Case $u>v$}
Suppose that $P_j\in \T_{j,u}$ and $P_i\in\T_{i,u}$ with $i\neq j$. We need to show that $P_{j,v}\cap P_{i,v}=\emptyset$. Indeed, suppose for the sake of contradiction that for instance $P_{j,v}\leq P_{i,v}$ so that  $\omega_{P_{i,v}}\subset \omega_{P_{j,v}}$.  Since $P_i\neq P_j$ we actually have $\omega_{P_i} \subset \omega_{P_{j,v}}$, where we use the fact that $\abs{\omega_{P_i}}<\abs{\omega_{P_j}}$ implies $\abs{\omega_{P_i}}\leq\tfrac14\abs{\omega_{P_j}}$ by the splitting performed in \eqref{eq:splitLength}. Thus we have
\begin{align*}
\omega_{T_i}\subset \omega_{P_i} \subset \omega_{P_{j,v}}\quad\mbox{and}\quad \omega_{T_{j,u}} \subset \omega_{P_{j,u}}.
\end{align*}
Using the previous inclusions and the fact that $s>v$ we get
\begin{align*}
c(\omega_{T_{i,u}})< \sup \omega_{P_{j,v}} \leq \inf \omega_{P_{j,u}} < c(\omega_{T_{j,u}})
\end{align*}
which means that the tree $\T_i$ was chosen first and thus $i<j$. However $P_{j,v}\leq P_{i,v}$ implies that $P_j\leq P_i \leq T_i$ so that the quartile $P_j$ qualified for the tree $\T_i$ but was not chosen which is a contradiction because of the maximality condition in the selection of $\T_i$.

Thus $\TT_u$ is $(u,v)$-good for all $u>v$. Using the definition of the quartile type we get  for all $u>v$ and $s=0,1$ that
\begin{align*}
 \sum_{j}   \NOrm \sum_{P\in \T^u_{j,u}(s)}\ip f,w_{P_{v}},w_{P_{v}} .L^{q}(\RR;X). ^{q}\lesssim \norm f.L^{q}(\RR;X). ^{q}.  
\end{align*}
Summing over the finitely many choices of $u>v$ and $s=0,1$ gives the desired estimate for $\sum_j |I_{T_{u,j}}|$ for $u>v$.

\subsubsection*{Case $u<v$}
We use a similar argument to show that $\TT_u$ is $(u,v)$-good also in this case. Suppose to the contrary that there are  $P_j\in \T_{j,u}$ and $P_i\in \T_{i,u}$ with $i\neq j$ and  $P_{j,v}\leq P_{i,v}$ so that  $\omega_{P_{i,v}}\subset \omega_{P_{j,v}}$.  Since $P_i\neq P_j$  and the splitting \eqref{eq:splitLength} is in force, we actually have $\omega_{P_i} \subset \omega_{P_{j,v}}$, and then
\begin{align*}
\omega_{T_i}\subset \omega_{P_i} \subset \omega_{P_{j,v}}\quad\mbox{and}\quad \omega_{T_{j,u}} \subset \omega_{P_{j,u}}.
\end{align*}
We have
\begin{align*}
c(\omega_{\tilde T_{i,s}})> \inf (\omega_{P_{j,v}})\geq\sup (\omega_{P_{j,u}})>c(\omega_{ T_{j,u}}).
\end{align*}
Remembering that we have chosen the trees $\T_j$ to have \emph{maximal} $c(\omega_{ T_{j,u}})$ (now that $u<v$) we conclude that $i<j$. However we have $P_j\leq T_i$ as before which leads to a contradiction. Thus $\TT_u$ is $(u,v)$-good again, and by the definition of the quartile-type hypothesis we estimate $\sum_{j} |I_{T_{u,j}}|$ exactly as in the case $u>v$.

The desired collection of trees is thus  $\{\T_{j,u}\}_{u\neq v; j}$.

Finally we observe that both the collection $\P_{\operatorname{small}}$ as well as the trees constructed are convex. Indeed let $\T_2$ be any of the constructed trees and $P',P''\in\T_2$ with $P'\leq P\leq P''$ for some quartile $P$. Since $\P$ is convex we have that $P\in\P$. Suppose that $P\notin \T_2$. Since $P\leq P''\leq T_2$ we necessarily have that $P\in\T_1$ where the tree $\T_1$ was chosen earlier in the construction. Since $P'\leq P\leq T_1$ we conclude that $P'$ should have been included in $\T_1$, a contradiction. For $\P_{\operatorname{small}}$ let $P'\leq P \leq P''$ with $P',P''\in \P_{\operatorname{small}}$. Since the original collection $\P$ was convex we conclude that $P\in\P$. Now $P\notin \P_{\operatorname{small}}$ means that $P$ was selected in one of the constructed trees, say $\T$. However in this case we should have also $P'\in \T$ by the maximality of $\T$, a contradiction.
\end{proof}

\begin{corollary}\label{c.decomp} Let $\P$ be any finite, convex collection of quartiles. Suppose that $\norm f.q_v.=1$ and define the $\size_{v}(\P)$ with respect to $f_v$  for all $v\in\{1,2,3\}$. Then $\P$ admits the decomposition
\begin{align*}
\P=\bigcup_{n\in\mathbb Z} \bigcup_j \T_{n,j}\cup\P_{\operatorname{residual}},
\end{align*}
where $\T_{n,j}$ are convex trees,
\begin{align*}
\size_v(\T_{n,j}) \leq 2^\frac{n}{q_v}\norm f_v.q_v. \quad\mbox{for all}\quad v\in\{1,2,3\},\quad \sum_{j} |I_{T_{n,j}}| \leq C 2^{-n},
\end{align*}
and $\size_v(\P_{\operatorname{residual}})=0$ for all $v\in\{1,2,3\}$.
\end{corollary}

\begin{proof}
If $\P$ is any finite collection of quartiles, there exists some large $n$ such that
\begin{align*}
\size_v(\P)\leq  2^\frac{n}{q_v}  \norm f_v.q_v.  \quad\mbox{for all}\quad v\in\{1,2,3\}.
\end{align*}
The decomposition follows by iterative use of Size Lemma.
\end{proof}

\section{The restricted weak type estimates above the quartile types}\label{s.local}
We start by a preliminary result that will allow us to get an initial estimate for the $v$-size of any collection $\P$, independently of the size lemma. At each step of the proof we will then  choose between the competing estimates, one coming from the lemma below and one being a consequence of the size reduction of any collection caused by the size lemma.
\begin{lemma}\label{l.size<1} Let $\mathcal  J \subseteq \{ I \in \mathcal D: \inf _I Mf \leq \lambda \}$ be a finite collection of dyadic intervals. For any interval $I\in\mathcal J$ write $I=I_0=I_l\cup I_r$ where $I_l,I_r$ are the dyadic children of $I$. Let $K$ be a dyadic intervals $K$. Then
\begin{align*}
\Norm \sum_{\substack{I_P\in\mathcal{J}\\ I_P\subseteq K}} \ip f,h_{I_a},h_{I_b}.L^p(\RR;X). \lesssim \lambda |K|^\frac{1}{p},
\end{align*}
for any possible combination of $a,b\in\{0,l,r\}$. 
\end{lemma}

\begin{proof}
(This is essentially like  \cite[Lemma 7.2]{12020209}	).
We fix any combination $a,b\in\{0,l,r\}$ and set
	\begin{align*}
		\tilde f\eqdef \sum_{I\in \mathcal J} \ip f,h_{I_a},h_{I_b}.
	\end{align*}
Denote by $\mathcal J^*(K)$ the maximal elements $I\in\mathcal J$ with $I\subseteq K$. We have
\begin{equation}\label{eq:splitImpSqFn}
  (\ind_K-\Exp_K)\tilde{f}
  =\sum_{\substack{I\in\mathcal J\\ I_b\subseteq K}} \ip f,h_{I_a},h_{I_b} 
  =\sum_{\substack{I\in\mathcal J\\ I\subseteq K}} \ip f,h_{I_a},h_{I_b} 
     +\ind_{\{l,r\}}(b)\ind_K\sum_{\substack{I\in\mathcal J\\ I=K^{(1)}}} \ip f,h_{I_a},h_{I_b},
\end{equation}
and
\begin{equation*}
  \sum_{\substack{I\in\mathcal J\\ I\subseteq K}} \ip f,h_{I_a},h_{I_b} 
 =\sum_{J\in\mathcal J^*(K)}\sum_{\substack{J\in\mathcal J\\ I\subseteq J}} \ip f \ind_{\cup \mathcal J^*(K)},h_{I_a},h_{I_b}
 =T(f \ind_{\cup \mathcal J^*(K)}),
\end{equation*}
where $T$ is an operator of the form considered in Lemma~\ref{lem:HaarShift}. That Lemma showed the boundedness on $L^p(\R;X)$, from which $T:L^1(\R;X)\to L^{1,\infty}(\R;X)$ follows by the standard Calder\'on--Zygmund method. Hence
\begin{equation*}
  \|T(f \ind_{\cup \mathcal J^*(K)})\|_{L^{1,\infty}(\R;X)}
  \lesssim\|f \ind_{\cup \mathcal J^*(K)}\|_{L^{1}(\R;X)}
  =\sum_{J\in\mathcal J^*(K)}\|f\ind_J\|_{L^{1}(\R;X)}
  \leq\sum_{J\in\mathcal J^*(K)}\abs{J}\inf_J Mf\leq\lambda\abs{K}.
\end{equation*}
The last sum in \eqref{eq:splitImpSqFn} has at most one term, for which the analogous estimate is immediate from
\begin{equation*}
  \abs{\ip f, h_{I_a}, h_{I_b}}
  \lesssim\abs{I}^{1/2}\fint_I \abs{f}dx\frac{1}{\abs{I}^{1/2}}\lesssim\lambda,
\end{equation*}
since $I\in\mathcal J$.

This shows that
\begin{equation*}
  \| (\ind_K-\Exp_K)\tilde{f}\|_{L^{1,\infty}(\R;X)}\lesssim\lambda\abs{K}
\end{equation*}
for all dyadic $K$, a weak formulation of the BMO condition, which by the John--Nirenberg--Str\"omberg inequalities bootstraps to
\begin{equation*}
  \| (\ind_K-\Exp_K)\tilde{f}\|_{L^{p}(\R;X)}\lesssim\lambda\abs{K}^{1/p},
\end{equation*}
which was the claim.
\end{proof}

\begin{corollary}\label{cor:impSqFn}
Let $\mathcal  J \subseteq \{ I \in \mathcal D: \inf _I Mf \leq \lambda \}$ be a finite collection of dyadic intervals, let $\T$ be a $u$-tree, and $v\in\{0,1,2,3\}\setminus\{u\}$.
Then
\begin{align*}
\Norm \sum_{\substack{P\in\T\\ I_P\in\mathcal J}}\ip f,w_{P_v}, w_{P_v}.L^p(\RR;X). \lesssim \lambda |K|^\frac{1}{p}.
\end{align*}
\end{corollary}

\begin{proof}
We have $\ip f,w_{P_v}, w_{P_v}=w^\infty_{T_u}\ip fw^\infty_{T_u}, w_{I_P,m}, w_{I_P,m}$ for $m=m(u,v)\in\{1,2,3\}$. The claim is then immediate from the previous lemma and Remark~\ref{rem:WalshVsHaar}.
\end{proof}


In the current and the following section we prove \emph{generalized restricted weak-type} estimates for the trilinear form $\Lambda$ in the spirit of \cite{MR1887641} and \cite{Thiele:book}. Let $\alpha=(\alpha_1,\alpha_2,\alpha_3)$ be a triple of real numbers with $a_v\leq 1$ for all $v\in\{1,2,3\}$. We will say that $\Lambda$ is of generalized restricted weak type $\alpha$ if for all triples $(E_1,E_2,E_3)$ of measurable sets of finite measure there is an index $j\in\{1,2,3\}$ and a subset $E_j '\subset E_j$ with $| E_j '|\geq \frac{1}{2}|E_j|$, that is \emph{a major subset}, such that
\begin{align*}
	|\Lambda(f_1,f_2,f_3)|\lesssim | E_1|^{\alpha_1} | E_2 |^{\alpha_2} |E_3|^{\alpha_3},
\end{align*} 
for all measurable functions $f_1,f_2,f_3$ with $|f_v|\leq \ind_{E' _v}$ for all $v\in\{1,2,3\}$. Here we have set $E'  _u\eqdef E_u$ for $u\in\{1,2,3\}\setminus\{j\}$.

The reader is referred to \cite{MR1887641} and \cite{Thiele:book} for more details. A general guideline to keep in mind is the following. Suppose that the trilinear form $\Lambda$ is of generalized restricted weak type $\beta$, with $\sum_v \beta_v=1$, for all $\beta$ in a neighborhood of a point $(\alpha_1,\alpha_2,\alpha_3)$ with $\sum_v \alpha_v=1$. If $\alpha_j>0$ for all $j\in\{1,2,3\}$ then $\Lambda$ satisfies the corresponding strong bounds:
\begin{align*}
	|\Lambda(f_1,f_2,f_3)|\lesssim \prod_{v=1} ^3 \norm f_v.L^{1/\alpha_j}(\RR;X_v).
\end{align*}
If exactly one of the exponents satisfies $\alpha_j\leq 0$, then the dual form in $j$ satisfies
\begin{align*}
	\norm B_j (f_u,f_\tau).L^{1/(1-\alpha_j)}(\RR;X_j).\lesssim \norm f_u .L^{1/{\alpha_u}}(\RR;X_u). \norm f_\tau.L^{1/{\alpha_\tau}}(\RR;X_\tau).,
\end{align*}
where $u,\tau\in\{1,2,3\}\setminus\{j\}$.

We first prove the local $L^{q_1}\times L^{q_2}$ estimates of the main theorem. Observe that this corresponds to the local $L^2$ case of the scalar theorem. By interpolation the following lemma implies Theorem \ref{t.main-local}.
\begin{lemma}\label{l.above} Suppose that the UMD Banach spaces $X_1,X_2, X_3$ have quartile types $q_1, q_2,q_3$ respectively with
\begin{align*}
\rho\eqdef \frac{1}{q_1}+\frac{1}{q_2}+\frac{1}{q_3}-1>0.
\end{align*}
 Let $E_1,E_2,E_3$ be  measurable sets of finite measure in $\R$, and let
\begin{align*}
0<\alpha_v<\frac{1}{q_v}\quad\mbox{for all}\quad v\in\{1,2,3\}	,\quad \sum_{v=1} ^3 \alpha_v=1.
\end{align*}
Let $E_\tau$ have maximal measure. There is a major subset $E_\tau '$ of $E_\tau$ such that
\begin{align*}
\abs{\Lambda(f_1,f_2,f_3)}\lesssim \prod_{v=1} ^3 |E_v|^{\alpha_v}
\end{align*}
whenever $|f_v|\leq \ind_{E' _v}$ for all $v\in\{1,2,3\}$. Here, $E_\tau '$ is specified and for $v\in\{1,2,3\}\setminus\{\tau\}$, we set $E_v ' = E_v$.
\end{lemma}

\begin{proof} Let $F$ be the set
\begin{align*}
F\eqdef \bigcup_{v=1} ^3 \big \{M \textbf 1_{E_v}>8|E_v|/|E_\tau| \big \}.
\end{align*}
Then $|F|\leq  |E_\tau|/2$ thus we can take $E_\tau ' \eqdef E_\tau \setminus F$ for our major subset. For $f_1,f_2,f_3$ as in the statement of the lemma we have
\begin{align*}
\Lambda(f_1,f_2,f_3)=\sum_{ P\in\P}\frac{1}{|I_P|^\frac{1}{2}}\ip f_1,w_{P_1},\cdot \ip f_2,w_{P_2}, \cdot \ip f_3,w_{P_3},=\sum_{\stackrel {P\in \P}{I_P\subseteq F}} +\sum_{\stackrel {P\in \P}{I_P\nsubseteq F}}.
\end{align*}
The first sum vanishes since $w_{P_\tau}$ is supported on $I_P\subseteq F$ and $f_\tau$ on $E_\tau '\subseteq F^c$. For the second sum we will need an additional upper bound on the size of any collection $\P'\subset \{P\in\P:I_P\nsubseteq F\}$. To estimate the $v$-size we fix some $v\in\{1,2,3\}$ and let $\T\subset \P'$ be any $u$-tree with $0\leq u \neq v\leq 3$. Let $T$ be the top of $\T$.

For any $P\in\T$ we have 
\begin{align*}
	\inf_{I_P}M(f_v w_{T_u } ^\infty)\leq \inf_{I_P}M(\textbf 1_{E_v})\leq 8 |E_v|/|E_\tau|,
\end{align*}
since $I_P\nsubseteq F$. By Corollary~\ref{cor:impSqFn} applied for $p=q_v$ and the previous two estimates we get
\begin{align*}
	\int\Abs{\sum_{P\in \T}\ip f_v,w_{P_v},w_{P_v}(x)}^{q_v} dx \lesssim \bigg( \frac{|E_v|}{|E_\tau|}\bigg)^{q_v}|I_T|.
\end{align*}
Thus $\size_v(\P')\lesssim \frac{|E_v|}{|E_\tau|} $ for all $v\in\{1,2,3\}$.

Hence
\begin{align*}
\Lambda(f_1,f_2,f_3)=\Lambda_{\P'}(f_1,f_2,f_3)\eqdef\sum_{ P\in \P'}\frac{1}{|I_P|^\frac{1}{2}}\ip f_1,w_{P_1},\cdot \ip f_2,w_{P_2}, \cdot \ip f_3,w_{P_3},.
\end{align*}
By Corollary \ref{c.decomp} and Lemma \ref{l.size<1} we can now estimate as follows:
\begin{align*}
\abs{\Lambda _{\P'}	 (f_1,f_2,f_3)} &\leq \sum_{n\in\mathbb Z} \sum_j \abs{\Lambda_{\T_{n,j}}(f_1,f_2,f_3)}  \leq  \sum_{n\in\mathbb Z} \sum _j  |I_{T_{n,j}}| \prod_{v=1} ^3 \size_v (\T_{n,j})
\\
&\lesssim \sum_{n\in\mathbb Z} 2^{-n} \prod_{v=1} ^3 \min(\frac{|E_v|}{|E_\tau|},2^\frac{n}{q_v}\norm f_v.L^{q_v}(\RR;X). )
\\
&=\sum_{n\in\mathbb Z} 2^{-n} \prod _{v=1} ^3 \min(\frac{|E_v|}{|E_\tau|},2^\frac{n}{q_v} |E_v|^\frac{1}{q_v})
\\
&\leq 	 \prod_{v=1} ^3\frac{ |E_v|}{|E_\tau|} \sum_{2^{-n}\leq \delta } 2^{-n} +\prod _{v=1} ^3  |E_v|^\frac{1}{q_v} \sum_{2^n\leq\delta^{-1}} 2^{n\rho},
\end{align*}
where $\delta$ is some real number to be chosen later and we have replaced $\rho= q_1 ^{-1}+q_2 ^{-1}+q_3 ^{-1}-1$. By our hypotheses we have $0<\rho\leq \frac{1}{2}$ where $\rho=\frac{1}{2}$ corresponds to the Hilbert-space case $q_1=q_2=q_3=2$. Using that $|E_v|\leq |E_\tau|$ and $\frac{1}{q_v}>\alpha_v$ for all $v\in\{1,2,3\}$ we get
\begin{align*}
\abs{\Lambda _{\P_1}	 (f_1,f_2,f_3)} &\lesssim \prod_{v=1} ^3 |E_v|^{\alpha_v} \bigg(\delta \prod_{v=1} ^3\frac{|E_v|^{1-\alpha_v}}{|E_\tau|}+\delta^{-\rho} \prod_{v=1} ^3 |E_v|^{\frac{1}{q_v}-\alpha_v}\bigg)
\\
&\leq \prod_{v=1} ^3 |E_v|^{\alpha_v} \bigg(\frac{\delta}{|E_\tau|} +\delta^{-\rho}|E_\tau|^\rho\bigg).
\end{align*}
The obvious choice $\delta\simeq |E_\tau|$ now gives the desired estimate.
\end{proof}

\begin{remark} Observe that our estimates and choice of $\delta$ were convenient but by no means optimal. This space for improvement will be exploited in the next section were we prove estimates `below the quartile types'.
\end{remark}

\section{The restricted weak type below the quartile types}   \label{s.below}
Obtaining estimates for the trilinear operator `below' the quartile type requires some additional work, as is the case for scalar valued functions and quartile types equal to $2$. Some additional attention is necessary in the vector-valued case since the three Banach spaces have different quartile types $q_1,q_2,q_3$. Our main estimate is the following lemma:

\begin{lemma} Suppose that the UMD Banach spaces $X_1,X_2,X_3$ have quartile types $q_1,q_2,q_3$ respectively. Let $E_1,E_2, E_3$ be measurable sets of finite measure and assume that
\begin{align*}
\bigg(\frac{|E_b|}{|E_\tau|}\bigg)^{q_b-1}\leq \bigg(\frac{|E_a|}{|E_\tau|}\bigg)^{q_a-1} \leq \bigg(\frac{|E_\tau|}{|E_\tau|}\bigg)^{q_\tau-1}=1
\end{align*}
where $a,b,\tau\in\{1,2,3\}$ are pairwise different. Note that in particular $E_\tau$ has maximal measure.  Let $\sum_{v=1} ^3 \beta_v=1$ and
\begin{align*}
\beta_b&<1,
\\
\frac{\beta_b}{q_b-1}&\leq\rho+ \frac{1}{q_b(q_b-1)},
\\
\frac{\beta_b}{q_b-1}+\frac{\beta_a}{q_a-1}&\leq\rho+ \frac{1}{q_b(q_b-1)}+\frac{1}{q_a(q_a-1)}.
\end{align*}
As before we assume that $\rho\eqdef \frac{1}{q_1}+\frac{1}{q_2}+\frac{1}{q_3}-1>0$. Then there is a major subset $E_\tau '\subset E_\tau$ such that
\begin{align*}
|\Lambda(f_1,f_2,f_3)|\lesssim |E_a|^{\beta_a} |E_b|^{\beta_b} |E_\tau|^{\beta_\tau},
\end{align*}
for all $|f_v|\leq \ind_{E' _v}$. Here we have set $E' _v =E_v$ for $v\in\{1,2,3\}\setminus\{\tau\}.$
\end{lemma}
\begin{proof} With the notations as in the proof of Lemma \ref{l.above} we have the main estimate
\begin{align*}
|\Lambda_{P'} (f_1,f_2,f_3)| \lesssim \sum_{n\in\mathbb Z} 2^{n} \prod _{v=1} ^3 \min(\frac{|E_v|}{|E_\tau|},2^{-\frac{n}{q_v}} |E_v|^\frac{1}{q_v}).
\end{align*}
Remember that $E_\tau$ has maximal measure. We now estimate the sum more carefully. For this it will be helpful to define
\begin{align*}
d_v\eqdef \bigg(\frac{|E_\tau|}{|E_v|}\bigg)^{q_v-1}|E_\tau|,\quad v\in\{1,2,3\}.
\end{align*}
Observe that our hypothesis translates to $d_b\geq d_a\geq d_\tau=|E_\tau|$. Also for any $v\in\{1,2,3\}$ we have that
\begin{align*}
\frac{|E_v|}{|E_\tau|}\leq 2^{-\frac{n}{q_v}} |E_v|^\frac{1}{q_v}\Leftrightarrow 2^n \leq d_v.
\end{align*}

We now split the sum according to the optimal value in the minimum:
\begin{align*}
|\Lambda_{P'}|\leq \sum_{2^n \leq d_\tau}+\sum_{d_\tau\leq 2^n\leq d_a}+\sum_{d_a\leq 2^n \leq d_b}+\sum_{d_b\leq 2^n}	\eqdef I+II+III+IV.
\end{align*}
The first term is the simplest:
\begin{align*}
	I\simeq \frac{|E_a|}{|E_\tau|}\frac{|E_b|}{|E_\tau|}d_\tau=\frac{|E_a||E_b|}{|E_\tau|}\leq \min (|E_a|, |E_b|)\leq|E_1|^{\alpha_1}|E_2|^{\alpha_2} |E_3|^{\alpha_3}.
\end{align*}
For $II$ we have
\begin{align*}
	II&\simeq   \frac{|E_a|}{|E_\tau|}\frac{|E_b|}{|E_\tau|} |E_\tau|^\frac{1}{q_\tau}d_a ^{1-\frac{1}{q_\tau}}=\frac{|E_a|}{|E_\tau|}\frac{|E_b|}{|E_\tau|} |E_\tau|^\frac{1}{q_\tau}\bigg(\frac{|E_\tau|}{|E_a|}\bigg)^{(q_a-1)(1-\frac{1}{q_\tau})}|E_\tau|^{1-\frac{1}{q_\tau}}
\\
&=\bigg(\frac{|E_b|}{|E_\tau|}\bigg)^{(q_b-1)\frac{1}{q_b - 1} } \bigg(\frac{|E_a|}{|E_\tau|}\bigg)^{(q_a-1)(\frac{1}{q_a-1}+\frac{1}{q_\tau}-1)} |E_\tau|
\end{align*}
Thus 
\begin{align*}
	II& \lesssim |E_a|^{\beta_a}|E_b|^{\beta_b}|E_\tau|^{\beta_\tau}\Leftrightarrow II\lesssim \bigg(\frac{|E_b|}{|E_\tau|}\bigg)^{(q_b-1)\frac{\beta_b}{q_b-1}}\bigg(\frac{|E_a|}{|E_\tau|}\bigg)^{(q_a-1)\frac{\beta_a}{q_a-1}}|E_\tau|
	\\
	&\Leftrightarrow \begin{cases} \beta_b \leq 1, \quad\mbox{\small and}\\  \frac{\beta_a}{q_a-1}+\frac{\beta_b}{q_b-1} \leq \frac{1}{q_a-1}+\frac{1}{q_b-1}+\frac{1}{q_\tau}-1.\end{cases}
\end{align*}
The estimate for $III+IV$ is slightly different according to the value of $\frac{1}{q_a}+\frac{1}{q_\tau}$. Observe that we always have $\frac{1}{q_a}+\frac{1}{q_\tau}\leq 1$ with equality if and only if $q_a=q_\tau=1/2$ which corresponds to the case that both $X_a$ and $X_\tau$ are Hilbert spaces.

\subsubsection*{Case $\frac{1}{q_a}+\frac{1}{q_\tau}<1$:} Now
\begin{align*}
	III+IV&\simeq \frac{|E_b|}{|E_\tau|} |E_a|^\frac{1}{q_a} |E_\tau|^\frac{1}{q_\tau}  d_b ^{1-\frac{1}{q_a}-\frac{1}{q_\tau}}
+|E_a|^\frac{1}{q_a}|E_b|^\frac{1}{q_b}|E_\tau|^\frac{1}{q_\tau} d_b ^{1-\frac{1}{q_a}-\frac{1}{q_b}-\frac{1}{q_\tau}}
\\
&= \bigg(\frac{|E_a|}{|E_\tau|}\bigg)^\frac{1}{q_a}\bigg( \frac{|E_b|}{|E_\tau|}\bigg)^{1-(q_b-1)(1-\frac{1}{q_a}-\frac{1}{q_\tau})}|E_\tau|
+|E_a|^\frac{1}{q_a}|E_b|^\frac{1}{q_b}	\bigg(\frac{|E_\tau|}{|E_b|}\bigg)^{(q_b-1)(1-\frac{1}{q_a}-\frac{1}{q_b}-\frac{1}{q_\tau})}|E_\tau|^{1-\frac{1}{q_a}-\frac{1}{q_b} }
\\
&\simeq  \bigg(\frac{|E_a|}{|E_\tau|}\bigg)^{(q_a-1)\frac{1}{q_a(q_a-1)}} \bigg(\frac{|E_b|}{|E_\tau|}\bigg)^{(q_b-1)(\frac{1}{q_b-1}+\frac{1}{q_a}+\frac{1}{q_\tau}-1)}|E_\tau|	.
\end{align*}
So we have
\begin{align*}
III+IV & \lesssim |E_a|^{\beta_\alpha} |E_b|^{\beta_b}|E_\tau|^{\beta_\tau}\Leftrightarrow III+IV\lesssim \bigg(\frac{|E_b|}{|E_\tau|}\bigg)^{(q_b-1)\frac{\beta_b}{q_b-1}}\bigg(\frac{|E_a|}{|E_\tau|}\bigg)^{(q_a-1)\frac{\beta_a}{q_a-1}}|E_\tau|
\\
	&\Leftrightarrow\begin{cases} \frac{\beta_b}{q_b-1} \leq  \frac{1}{q_b-1}+\frac{1}{q_a}+\frac{1}{q_\tau}-1, \quad\mbox{\small and}
\\ 
 \frac{\beta_a}{q_a-1}+\frac{\beta_b}{q_b-1}\leq  \frac{1}{q_b-1}+\frac{1}{q_a-1} +\frac{1}{q_\tau}-1.\end{cases}
\end{align*}
Observe that since we assume $\frac{1}{q_a}+\frac{1}{q_\tau}-1<0$  the first condition contains the condition $\beta_b<1$. Also the second condition is the same as the second condition for $II$.
\subsubsection*{Case $\frac{1}{q_a}+\frac{1}{q_\tau}=1$:} Now
\begin{align*}
III&\simeq \frac{|E_b|}{|E_\tau|} |E_a|^\frac{1}{q_a} |E_\tau|^\frac{1}{q_\tau} \log_2 \frac{d_b}{d_a}=\frac{|E_b|}{|E_\tau|} |E_a|^\frac{1}{q_a} |E_\tau|^\frac{1}{q_\tau} \log_2 \bigg[\bigg	(\frac{|E_\tau|}{|E_b|}\bigg)^{q_b-1} \bigg(\frac{|E_a|}{|E_\tau|}\bigg)^{q_a-1}  \bigg]
\\
&\lesssim_\epsilon \bigg(\frac{|E_b|}{|E_\tau|} \bigg)^{(q_b-1)(\frac{1}{q_b-1} -\epsilon)} \bigg(\frac{|E_a|}{|E_\tau|}\bigg)^{(q_a-1)(\frac{1}{q_a(q_a-1)}+\epsilon)} |E_\tau|.
\end{align*}
The last expression is dominated by $|E_a|^{\beta_a}|E_b|^{\beta_b}|E_\tau|^{\beta_\tau}$ if 
\begin{align*}
\begin{cases} \beta_b<1, \quad\mbox{and}
\\ 
 \frac{\beta_a}{q_a-1}+\frac{\beta_b}{q_b-1}\leq  \frac{1}{q_b-1}+\frac{1}{q_a-1} +\frac{1}{q_\tau}-1.\end{cases}
\end{align*}
For $IV$ we get the same estimate as in the case $\frac{1}{q_a}+\frac{1}{q_\tau}<1$.
\end{proof}
By considering all permutations in the hypothesis of the previous Lemma we immediately get:
\begin{corollary}\label{c.main}Suppose that the UMD Banach spaces $X_1,X_2,X_3$ have quartile type $q_1,q_2,q_3$ respectively, with
\begin{align*}
\rho=\frac{1}{q_1}+\frac{1}{q_2}+\frac{1}{q_3}-1>0.
\end{align*}
Let $E_1,E_2,E_3$ be measurable subsets of the real line of finite measure and assume that $E_\tau$ has maximal measure. Suppose that
 $\sum_v \beta_v=1$ and that for all $v,u\in\{1,2,3\}$ with $v\neq u$ we have
\begin{align*}
\beta_v&<1, 
\\
\frac{\beta_v}{q_v-1}&\leq\rho+ \frac{1}{q_v(q_v-1)},
\\
\frac{\beta_v}{q_v-1}+\frac{\beta_u}{q_u-1}&\leq\rho+ \frac{1}{q_u(q_u-1)}+\frac{1}{q_u(q_u-1)}.
\end{align*}
Then there is a major subset $E_\tau '\subset E_\tau$ such that
\begin{align*}
|\Lambda(f_1,f_2,f_3)|\lesssim |E_1|^{\beta_1} |E_2|^{\beta_2} |E_3|^{\beta_3},
\end{align*}
whenever $|f_v|\leq \ind_{E_v {'}}$. Again we have set  $E{'} _v =E_v$ for $v\in\{1,2,3\}\setminus\{\tau\}.$
\end{corollary}
The conditions of Corollary \ref{c.main} are illustrated in Figures \ref{f.triangle}, \ref{f.general}. Corollary \ref{c.main} says that for parameters $\beta_1,\beta_2,\beta_3$ with $\sum_v \beta_v	=1$ and $(\beta_1,\beta_2)$ in the open hexagon $ABCDEF$, the trilinear form $\Lambda$ is of generalized restricted weak type $(\beta_1,\beta_2,\beta_3)$. Observe that the triangle $c$ corresponds to the local $L^{q_1}\times L^{q_2}$ estimates of Lemma \ref{l.above} which, for the scalar case, correspond to the local $L^2$ estimates. One simple way to determine the points $A,B,C,D,E$ and $F$ is to extend the sides of the local triangle $c$ until they cross the restriction lines $\beta_1=\frac{1}{q_1}+(q_1-1)\rho$, $\beta_2=\frac{1}{q_2}+(q_2-1)\rho$ and $\frac{\beta_1+\beta_2}{q_3-1}=\frac{1}{q_3}-\rho$. We get
\begin{align*}
	A:&\quad(\frac{1}{q_1}-\rho q_3,\frac{1}{q_2},\frac{1}{q_3}+\rho q_3-\rho),\quad\quad D:\quad (\frac{1}{q_1}+\rho q_1-\rho,\frac{1}{q_2}, \frac{1}{q_3}-\rho q_1),
	\\
	B:&\quad(\frac{1}{q_1},\frac{1}{q_2}-\rho q_3,\frac{1}{q_3}+\rho q_3 - \rho),\quad\quad E:\quad (\frac{1}{q_1},\frac{1}{q_2}+\rho q_2-\rho,\frac{1}{q_3}-\rho q_2),
	\\
	C:&\quad (\frac{1}{q_1}+\rho q_1-\rho, \frac{1}{q_2}-\rho q_1, \frac{1}{q_3}),\quad\quad F:\quad (\frac{1}{q_1}-\rho q_2, \frac{1}{q_2}+\rho q_2-\rho,\frac{1}{q_3}).
\end{align*}

Interpolating the generalized restricted weak type estimates as in \cite[Theorems 3.2, 3.6 and 3.8]{Thiele:book},
Corollary \ref{c.main} implies our main theorem:

\begin{theorem}\label{thm:main} Suppose that $\sum_{v=1} ^3\beta_v=1$ and $(\beta_1,\beta_2)\in ABCDEF$ where $ABCDEF$ denotes the open convex hexagon of Figure \ref{f.general}.
	
\noindent{(i)} If $\beta_v>0$ for all $v=1,2,3$ then $\Lambda$ satisfies the strong bounds
\begin{align*}
	|\Lambda(f_1,f_2,f_3)|\lesssim \prod_{v=1} ^3 \|f\|_{L^{1/\beta_v} (\RR:X_v)}.
\end{align*}

\noindent{(ii)} If $\beta_v\leq 0$ for one $v\in\{1,2,3\}$ then $\beta_{u},\beta_{\tau}>0$ for $u,\tau\in\{1,2,3\}\setminus\{ v\}$ and the bilinear quartile operator satisfies
\begin{align*}
	\norm {B_v}(f_u,f_{\tau}).L^{1/(1-\beta_v)}(\RR;X_v).   \lesssim \norm f_u.L^{1/ \beta_u}(\RR;X_u). \norm f_\tau.L^{1/\beta_\tau}(\RR;X_\tau).
\end{align*}
\end{theorem}

\begin{figure}[htb]
\centering
 \def\svgwidth{350pt}
\begingroup%
  \makeatletter%
  \providecommand\color[2][]{%
    \errmessage{(Inkscape) Color is used for the text in Inkscape, but the package 'color.sty' is not loaded}%
    \renewcommand\color[2][]{}%
  }%
  \providecommand\transparent[1]{%
    \errmessage{(Inkscape) Transparency is used (non-zero) for the text in Inkscape, but the package 'transparent.sty' is not loaded}%
    \renewcommand\transparent[1]{}%
  }%
  \providecommand\rotatebox[2]{#2}%
  \ifx\svgwidth\undefined%
    \setlength{\unitlength}{794.77019043bp}%
    \ifx\svgscale\undefined%
      \relax%
    \else%
      \setlength{\unitlength}{\unitlength * \real{\svgscale}}%
    \fi%
  \else%
    \setlength{\unitlength}{\svgwidth}%
  \fi%
  \global\let\svgwidth\undefined%
  \global\let\svgscale\undefined%
  \makeatother%
  \begin{picture}(1,0.58287471)%
    \put(0,0){\includegraphics[width=\unitlength]{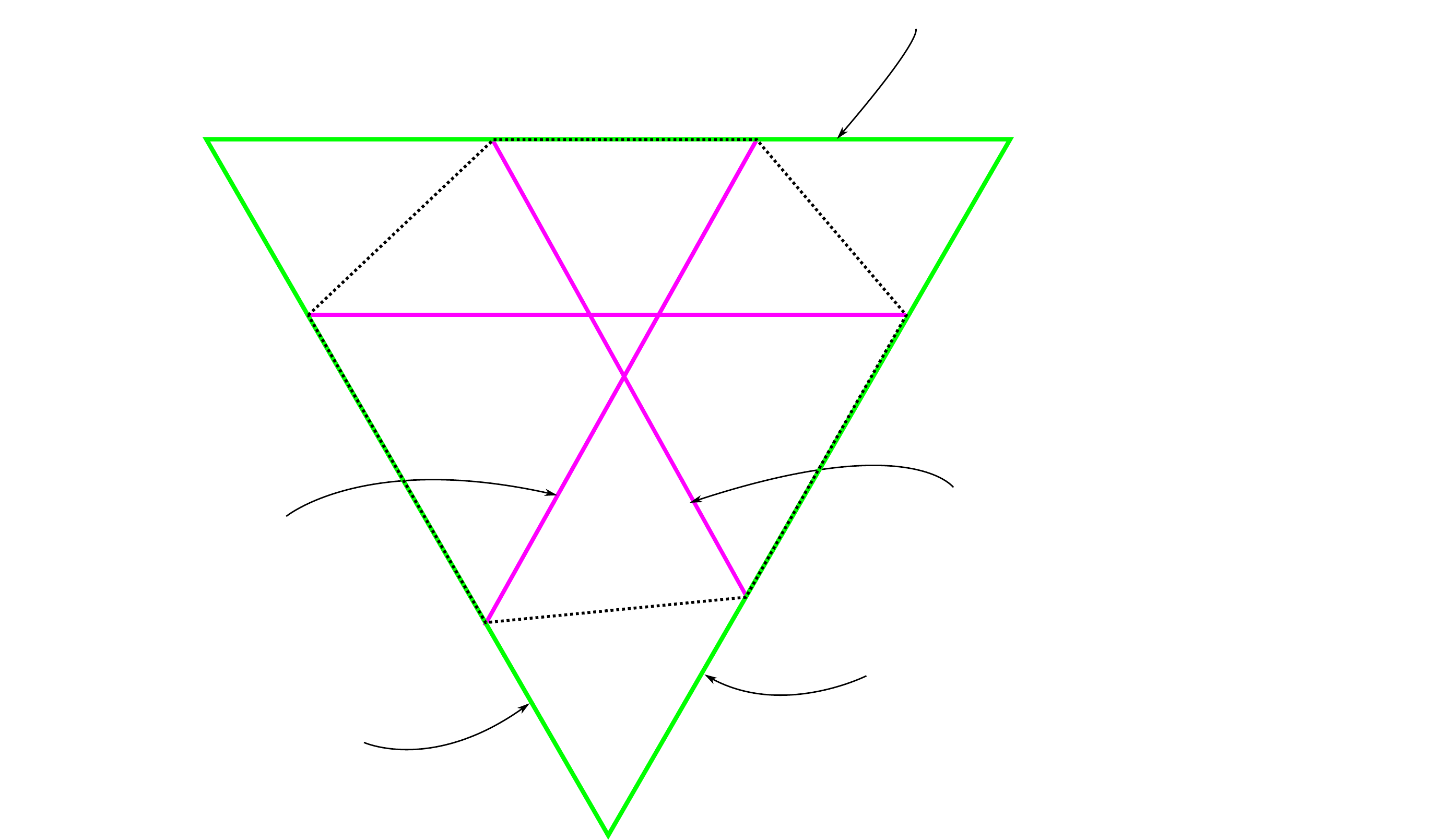}}%
    \put(0.33073833,0.50191826){\color[rgb]{0,0,0}\makebox(0,0)[lb]{\smash{$C$}}}%
    \put(0.50886402,0.50275848){\color[rgb]{0,0,0}\makebox(0,0)[lb]{\smash{$D$}}}%
    \put(0.65002026,0.35404031){\color[rgb]{0,0,0}\makebox(0,0)[lb]{\smash{$E$}}}%
    \put(0.53575094,0.15911034){\color[rgb]{0,0,0}\makebox(0,0)[lb]{\smash{$F$}}}%
    \put(0.30049057,0.13894512){\color[rgb]{0,0,0}\makebox(0,0)[lb]{\smash{$A$}}}%
    \put(0.16892624,0.35151969){\color[rgb]{0,0,0}\makebox(0,0)[lb]{\smash{$B$}}}%
    \put(0.42652291,0.34395774){\color[rgb]{0,0,0}\makebox(0,0)[lb]{\smash{$c$}}}%
    \put(0.2584798,0.3742055){\color[rgb]{0,0,0}\makebox(0,0)[lb]{\smash{$\beta_1=\frac{1}{q_1}$}}}%
    \put(0.05273608,0.20434871){\color[rgb]{0,0,0}\rotatebox{0.12012814}{\makebox(0,0)[lb]{\smash{$\beta_2=\frac{1}{q_2}$}}}}%
    \put(0.63151974,0.2244803){\color[rgb]{0,0,0}\makebox(0,0)[lb]{\smash{$\beta_3=\frac{1}{q_3}$}}}%
    \put(0.61053011,0.11373866){\color[rgb]{0,0,0}\makebox(0,0)[lb]{\smash{$\beta_2=\frac{1}{q_2}+(q_2-1)\rho$}}}%
    \put(-0.00167108,0.0743167){\color[rgb]{0,0,0}\makebox(0,0)[lb]{\smash{$\beta_3=\frac{1}{q_3}+(q_3-1)\rho$}}}%
    \put(0.44892891,0.5675794){\color[rgb]{0,0,0}\makebox(0,0)[lb]{\smash{$\beta_1=\frac{1}{q_1}+(q_1-1)\rho$}}}%
  \end{picture}%
\endgroup%
\caption[Generalized restricted weak-type convex hull]{Generalized restricted weak-type for $\Lambda:X_1\times X_2\times X_3\to \C$ in the plane $\beta_1+\beta_2+\beta_3=1$.}\label{f.triangle}
\end{figure}
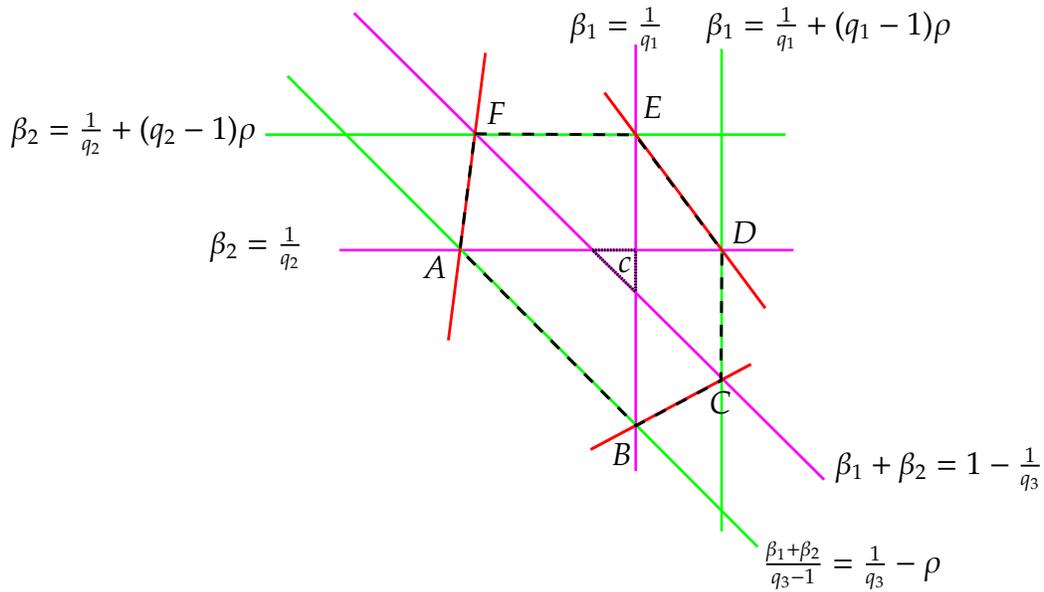
\begin{figure}[hbt]
\centering
 \def\svgwidth{530pt}
\begingroup%
  \makeatletter%
  \providecommand\color[2][]{%
    \errmessage{(Inkscape) Color is used for the text in Inkscape, but the package 'color.sty' is not loaded}%
    \renewcommand\color[2][]{}%
  }%
  \providecommand\transparent[1]{%
    \errmessage{(Inkscape) Transparency is used (non-zero) for the text in Inkscape, but the package 'transparent.sty' is not loaded}%
    \renewcommand\transparent[1]{}%
  }%
  \providecommand\rotatebox[2]{#2}%
  \ifx\svgwidth\undefined%
    \setlength{\unitlength}{378.29629517bp}%
    \ifx\svgscale\undefined%
      \relax%
    \else%
      \setlength{\unitlength}{\unitlength * \real{\svgscale}}%
    \fi%
  \else%
    \setlength{\unitlength}{\svgwidth}%
  \fi%
  \global\let\svgwidth\undefined%
  \global\let\svgscale\undefined%
  \makeatother%
  \begin{picture}(1,0.40444256)%
    \put(0,0){\includegraphics[width=\unitlength]{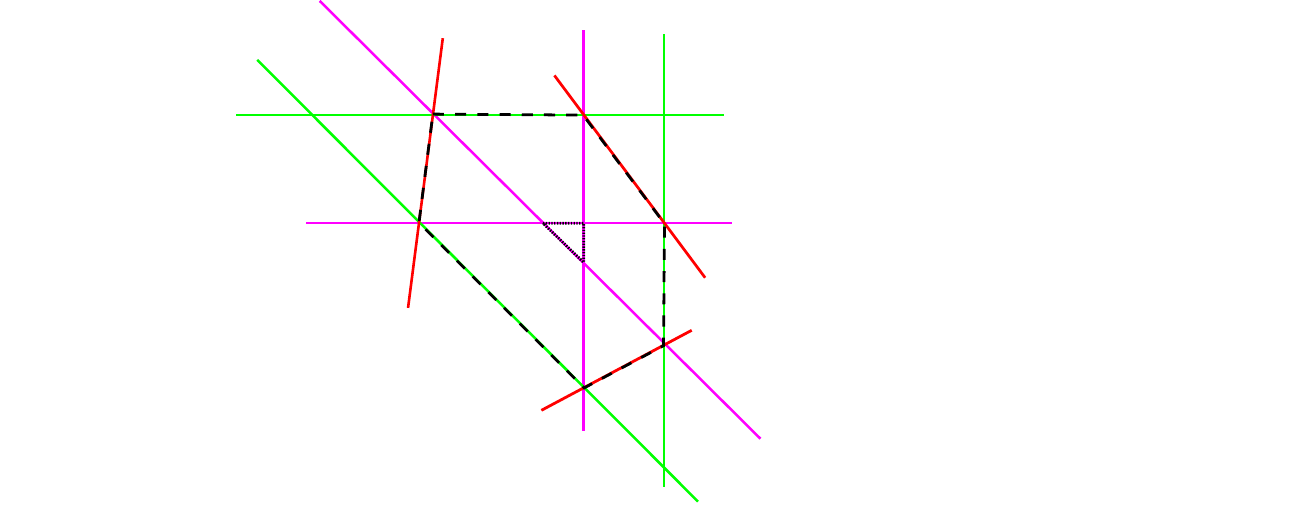}}%
    \put(0.14072331,0.23198879){\color[rgb]{0,0,0}\makebox(0,0)[lb]{\smash{$\beta_2=\frac{1}{q_2}$}}}%
    \put(0.39781825,0.38887342){\color[rgb]{0,0,0}\makebox(0,0)[lb]{\smash{$\beta_1=\frac{1}{q_1}$}}}%
    \put(0.58722231,0.0744167){\color[rgb]{0,0,0}\makebox(0,0)[lb]{\smash{$\beta_1+\beta_2=1-\frac{1}{q_3}$}}}%
    \put(0.53527664,0.00398993){\color[rgb]{0,0,0}\makebox(0,0)[lb]{\smash{$\frac{\beta_1+\beta_2}{q_3-1}=\frac{1}{q_3}-\rho$}}}%
    \put(-0.00105324,0.31385963){\color[rgb]{0,0,0}\makebox(0,0)[lb]{\smash{$\beta_2=\frac{1}{q_2}+(q_2-1)\rho$}}}%
    \put(0.2931384,0.21470904){\color[rgb]{0,0,0}\makebox(0,0)[lb]{\smash{$A$}}}%
    \put(0.42772677,0.08132912){\color[rgb]{0,0,0}\makebox(0,0)[lb]{\smash{$B$}}}%
    \put(0.49721124,0.11818613){\color[rgb]{0,0,0}\makebox(0,0)[lb]{\smash{$C$}}}%
    \put(0.51337392,0.23963292){\color[rgb]{0,0,0}\makebox(0,0)[lb]{\smash{$D$}}}%
    \put(0.44962949,0.32890524){\color[rgb]{0,0,0}\makebox(0,0)[lb]{\smash{$E$}}}%
    \put(0.3386054,0.32437372){\color[rgb]{0,0,0}\makebox(0,0)[lb]{\smash{$F$}}}%
    \put(0.43216362,0.21894007){\color[rgb]{0,0,0}\makebox(0,0)[lb]{\smash{$c$}}}%
    \put(0.4950965,0.38887342){\color[rgb]{0,0,0}\makebox(0,0)[lb]{\smash{$\beta_1=\frac{1}{q_1}+(q_1-1)\rho$}}}%
  \end{picture}%
\endgroup%
\caption[Generalized restricted weak-type convex hull]{Generalized restricted weak-type for $\Lambda:X_1\times X_2\times X_3\to \C$ in the $(\beta_1,\beta_2)$-plane.}\label{f.general}
\end{figure}

\begin{example} Suppose that $X$ is a UMD Banach space such that both $X,X^*$ have quartile type $2< q<4$. For example this is the case if $X$ is an interpolation space $X=[Y,H]_\theta$ between some UMD Banach space and a Hilbert space $H$. Define $X_1\eqdef X$, $X_2\eqdef X^*$ and $X_3\eqdef \C$. Observe that $q<4$ implies that $\rho=\frac{2}{q}-\frac{1}{2}>0.$ The trilinear form $\Pi:X\times X^*\times \C \to \C$ is defined in the obvious way:
\begin{align*}
\Pi(x,x^*,\lambda)=\lambda  \ip x,x^*,=\lambda x^*(x).
\end{align*}
It is not hard to see that the Theorem~\ref{thm:main} specialized to this case gives the following statements:

\noindent{(i)} If $2<q<3$ then the bilinear quartile operator maps $L^{p_1}(\RR;X)\times L^{p_2}(\RR;X^*)$ into $L^r(\RR;\C)$ whenever
\begin{align*}
\frac{1}{r}=\frac{1}{p_1}+\frac{1}{p_2},\quad \big(\frac{5-q}{2}-\frac{1}{q}\big)^{-1}<p_1,p_2\leq \infty,\quad \frac{2}{5-q}<r<\frac{q}{q-2}.
\end{align*}

Figure \ref{f.q<3} shows the convex hull of the conditions for $(\beta_1,\beta_2,\beta_3)$ in the $(\beta_1,\beta_2)$-plane in the case $2<q<3$. Remember that $\beta_3$ is then uniquely determined as $\beta_3=1-\beta_1-\beta_2$

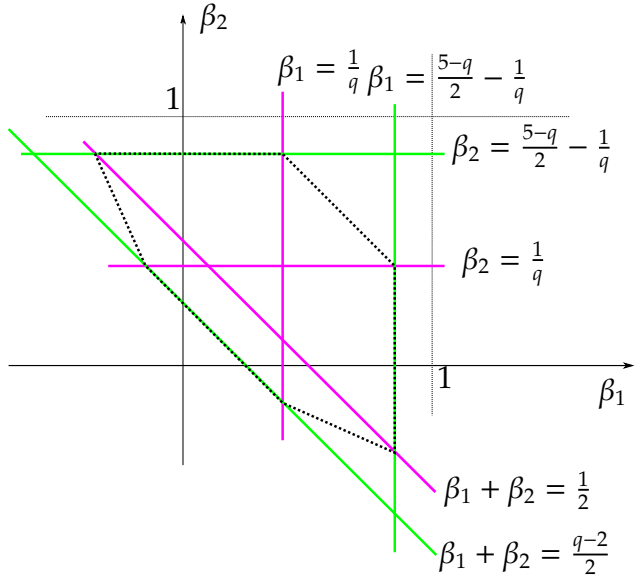
\begin{figure}[htb]
\centering
 \def\svgwidth{350pt}
\begingroup%
  \makeatletter%
  \providecommand\color[2][]{%
    \errmessage{(Inkscape) Color is used for the text in Inkscape, but the package 'color.sty' is not loaded}%
    \renewcommand\color[2][]{}%
  }%
  \providecommand\transparent[1]{%
    \errmessage{(Inkscape) Transparency is used (non-zero) for the text in Inkscape, but the package 'transparent.sty' is not loaded}%
    \renewcommand\transparent[1]{}%
  }%
  \providecommand\rotatebox[2]{#2}%
  \ifx\svgwidth\undefined%
    \setlength{\unitlength}{297.13157959bp}%
    \ifx\svgscale\undefined%
      \relax%
    \else%
      \setlength{\unitlength}{\unitlength * \real{\svgscale}}%
    \fi%
  \else%
    \setlength{\unitlength}{\svgwidth}%
  \fi%
  \global\let\svgwidth\undefined%
  \global\let\svgscale\undefined%
  \makeatother%
  \begin{picture}(1,0.61532706)%
    \put(0,0){\includegraphics[width=\unitlength]{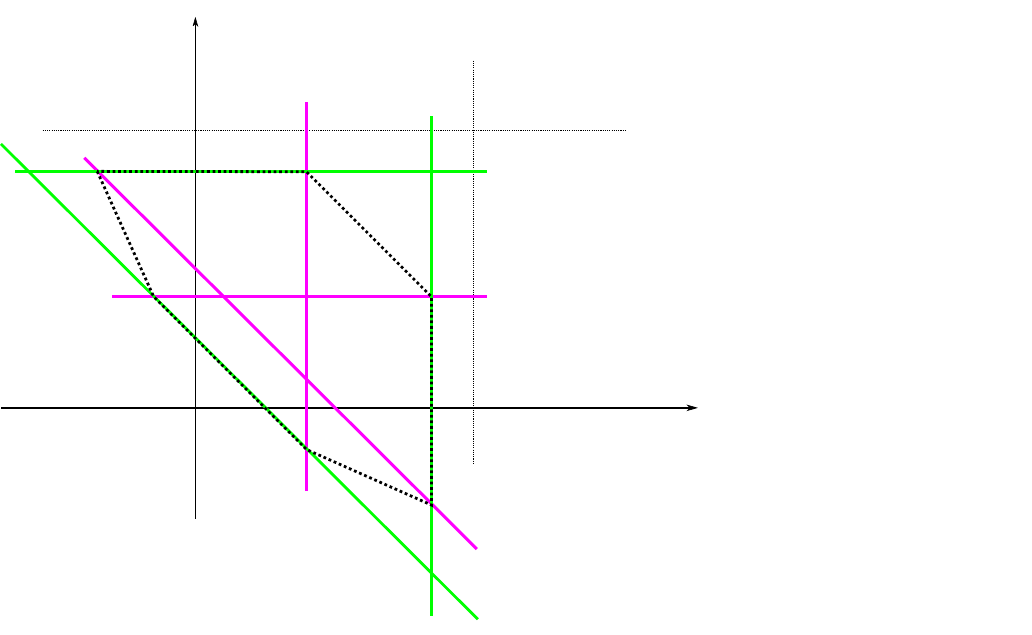}}%
    \put(0.20824107,0.58887203){\color[rgb]{0,0,0}\makebox(0,0)[lb]{\smash{$\beta_2$}}}%
    \put(0.17054738,0.49732999){\color[rgb]{0,0,0}\makebox(0,0)[lb]{\smash{$1$}}}%
    \put(0.63902667,0.18231811){\color[rgb]{0,0,0}\makebox(0,0)[lb]{\smash{$\beta_1$}}}%
    \put(0.46402001,0.19578025){\color[rgb]{0,0,0}\makebox(0,0)[lb]{\smash{$1$}}}%
    \put(0.49132873,0.32482354){\color[rgb]{0,0,0}\makebox(0,0)[lb]{\smash{$\beta_2=\frac{1}{q}$}}}%
    \put(0.29228271,0.53252364){\color[rgb]{0,0,0}\makebox(0,0)[lb]{\smash{$\beta_1=\frac{1}{q}$}}}%
    \put(0.47209724,0.07635243){\color[rgb]{0,0,0}\makebox(0,0)[lb]{\smash{$\beta_1+\beta_2=\frac{1}{2}$}}}%
    \put(0.46671242,0.00634978){\color[rgb]{0,0,0}\makebox(0,0)[lb]{\smash{$\beta_1+\beta_2=\frac{q-2}{2}$}}}%
    \put(0.38920947,0.521754){\color[rgb]{0,0,0}\makebox(0,0)[lb]{\smash{$\beta_1=\frac{5-q}{2}-\frac{1}{q}$}}}%
    \put(0.4807514,0.44636652){\color[rgb]{0,0,0}\makebox(0,0)[lb]{\smash{$\beta_2=\frac{5-q}{2}-\frac{1}{q}$}}}%
  \end{picture}%
\endgroup%
\caption[Generalized restricted weak-type convex hull]{Generalized restricted weak-type estimates for $\Lambda:X\times X^*\times \C\to \C$, $2<q<3$.}\label{f.q<3}  
\end{figure}

\noindent{(ii)} If $3\leq q<4$ then the bilinear quartile operator maps $L^{p_1}(\RR;X)\times L^{p_2}(\RR;X^*)$ into $L^r(\RR;\C)$
whenever the conditions in $(i)$ hold and in addition
\begin{equation*}
  \frac{q^2-3q+1}{q}<\min\Big(\frac{q-1}{p_1}+\frac{q-2}{p_2},\frac{q-2}{p_1}+\frac{q-1}{p_2}\Big).
\end{equation*}

Figure \ref{f.q>3} describes the domain of restricted weak type estimates for $3\leq q <4$. In this case the convex hexagon is strictly contained in $[0,1)^2$.

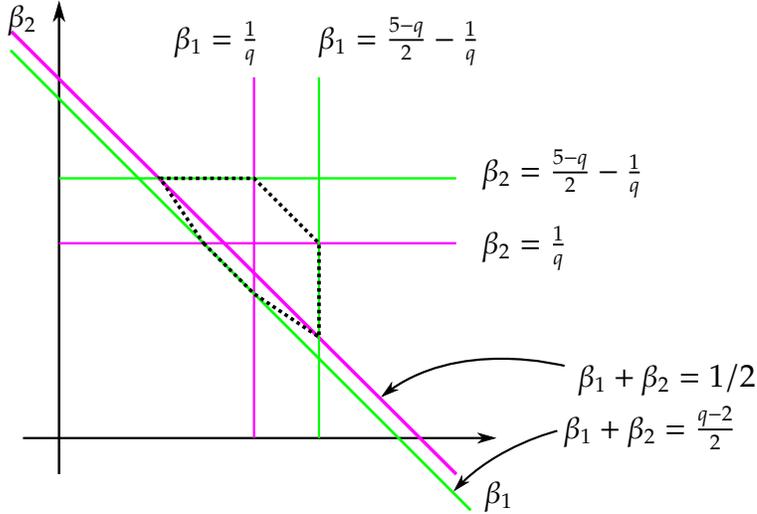
\begin{figure}[htb]
\centering
 \def\svgwidth{400pt}
\begingroup%
  \makeatletter%
  \providecommand\color[2][]{%
    \errmessage{(Inkscape) Color is used for the text in Inkscape, but the package 'color.sty' is not loaded}%
    \renewcommand\color[2][]{}%
  }%
  \providecommand\transparent[1]{%
    \errmessage{(Inkscape) Transparency is used (non-zero) for the text in Inkscape, but the package 'transparent.sty' is not loaded}%
    \renewcommand\transparent[1]{}%
  }%
  \providecommand\rotatebox[2]{#2}%
  \ifx\svgwidth\undefined%
    \setlength{\unitlength}{117.1828125bp}%
    \ifx\svgscale\undefined%
      \relax%
    \else%
      \setlength{\unitlength}{\unitlength * \real{\svgscale}}%
    \fi%
  \else%
    \setlength{\unitlength}{\svgwidth}%
  \fi%
  \global\let\svgwidth\undefined%
  \global\let\svgscale\undefined%
  \makeatother%
  \begin{picture}(1,0.48335656)%
    \put(0,0){\includegraphics[width=\unitlength]{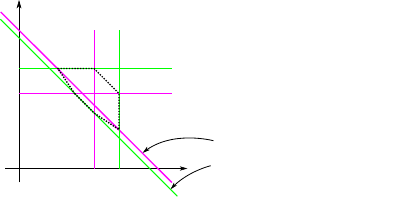}}%
    \put(-0.00106673,0.45834574){\color[rgb]{0,0,0}\makebox(0,0)[lb]{\smash{$\beta_2$}}}%
    \put(0.44951131,0.00776776){\color[rgb]{0,0,0}\makebox(0,0)[lb]{\smash{$\beta_1$}}}%
    \put(0.44743227,0.31247708){\color[rgb]{0,0,0}\makebox(0,0)[lb]{\smash{$\beta_2=\frac{5-q}{2}-\frac{1}{q}$}}}%
    \put(0.44743227,0.24420768){\color[rgb]{0,0,0}\makebox(0,0)[lb]{\smash{$\beta_2=\frac{1}{q}$}}}%
    \put(0.2924917,0.43786501){\color[rgb]{0,0,0}\makebox(0,0)[lb]{\smash{$\beta_1=\frac{5-q}{2}-\frac{1}{q}$}}}%
    \put(0.1559529,0.43786501){\color[rgb]{0,0,0}\makebox(0,0)[lb]{\smash{$\beta_1=\frac{1}{q}$}}}%
    \put(0.53801777,0.11797389){\color[rgb]{0,0,0}\makebox(0,0)[lb]{\smash{$\beta_1+\beta_2=1/2$}}}%
    \put(0.52460765,0.06921027){\color[rgb]{0,0,0}\makebox(0,0)[lb]{\smash{$\beta_1+\beta_2=\frac{q-2}{2}$}}}%
  \end{picture}%
\endgroup%
\caption[Restricted weak-type convex hull]{Generalized restricted weak-type estimates for $\Lambda:X\times X^*\times \C\to \C$, $3<q<4$.}\label{f.q>3}
\end{figure}
\end{example}

\bibliographystyle{plain}

 \end{document}